\documentclass[a4paper,10pt]{article}
\usepackage[margin=2.5cm]{geometry}
\usepackage[utf8]{inputenc}
\usepackage{amsmath,amsthm,amssymb}
\usepackage{empheq}

\usepackage{hyperref}
\hypersetup{
  colorlinks=true,
  linkcolor=black,
  citecolor=blue,
  urlcolor=blue,
  pdftitle={},
  pdfauthor={},
}
\newtheorem{thm}{Theorem}[section]
\newtheorem{lem}[thm]{Lemma}
\newtheorem{rem}[thm]{Remark}

\newtheorem{definition}[thm]{Definition}
\swapnumbers
\usepackage{xcolor}

\newcommand{\eps}{\varepsilon}
\newcommand{\R}{\mathbb{R}}
\newcommand{\N}{\mathbb{N}}

\newcommand{\io}{\int_\Omega}

\usepackage{graphicx}

\allowdisplaybreaks
\title{Global existence of weak solutions in a Keller-Segel system with local anisotropy and nonlinear diffusion}
\author{Max Hegenbart\\
\small{Leibniz Universität Hannover}\\
\small{Institut für Angewandte Mathematik}\\
\small{hegenbart@ifam.uni-hannover.de}}

\begin{document}

\maketitle

\begin{abstract}
We show global existence of weak solutions to the following system of partial differential equations with nonlinear diffusion
in both equations and anisotropic signal production
\begin{equation}\label{Hauptsystem Abstract}
 \begin{cases}
  u_t=\nabla\cdot((u+1)^{m_1-1}\nabla u)-\nabla\cdot(u\nabla v),\\
  v_t=\nabla\cdot((v+1)^{m_2-1}\nabla v)+\nabla\cdot(u\nabla v)-v+u,
 \end{cases}
\end{equation}
in a smooth bounded domain $\Omega\subseteq\R^n$, $n\in\N$ with nonnegative initial data and homogenuous Neumann
boundary conditions. Moreover, we show boundedness of $u$ in $L^\infty((0,\infty);\:L^{m_1}(\Omega))$ and
$v$ in $L^\infty(\Omega\times(0,\infty))$.
The strategy will be to show global existence and boundedness
of classical solutions $(u_\eps,v_\eps)$ to the approximating system
\begin{equation}\label{eps>0-System Abstract}
\begin{cases}
 u_{\eps t}=\nabla\cdot((u_\eps+1)^{m_1-1}\nabla u_\eps)
 -\nabla\cdot(\frac{u_\eps\nabla v_\eps}{1+\eps\vert\nabla v_\eps\vert})\\
 v_{\eps t}=\nabla\cdot((v_\eps+1)^{m_2-1}\nabla v_\eps)+\nabla\cdot(u_\eps\nabla v_\eps)-v_\eps+u_\eps.
\end{cases}
\end{equation}
Passing to the limit $\eps\searrow 0$, we obtain convergence of weak solutions of \eqref{eps>0-System Abstract} to
weak solutions of system \eqref{Hauptsystem Abstract}.
One of the main challenges lie in the proof of weak convergence of
$\frac{\nabla v_\eps}{1+\eps\vert\nabla v_\eps\vert}$ to $\nabla v$ in $L^2(\Omega\times(0,T))$ for $T\in(0,\infty)$.
\end{abstract}

\textbf{Math Subject Classification (MSC 2020):} 35K51 (primary);  	
35K59;   	
35B45;   	
35Q92;   	
92C17 

\section{Introduction}
The emergence of patterns generated by cells or particles, following gradients of a chemical signal,
plays an important role in biology, immunology, cancer and even social sciences \cite{Painter2019}.
Based on observations in experiments, there seems to be a large potential in processes where patterns of
limited size occur due to anisotropic signal production.
This means that only parts of the
particles are able to react with the surrounding chemical. Hence, the local density of the chemical
depends on the orientation of the particles \cite{PhysRevLett.115.258301}.
With asymmetrically coated particles there are far more promising applications
in biomedical applications than with
homogenuous particles. For exapmle, \textit{Janus particles} are used as artifical antigen-presenting cells
for T cell activation \cite{Bo2014}.
Although there are experimental results and observations, the mathematical analysis of pattern formations
in context of anisotropic chemoattractant signal production seems to be at an early stage.
We consider the following Neumann inital-boundary value problem
\begin{equation}\label{Poröse Medien mit nichtlinearer Diffusion}
 \begin{cases}
  u_t=\nabla\cdot((u+1)^{m_1-1}\nabla u)-\nabla\cdot(u\nabla v)&\quad\text{in }\Omega\times(0,T_{max}),\\
  v_t=\nabla\cdot((v+1)^{m_2-1}\nabla v)+\nabla\cdot(u\nabla v)-v+u&\quad\text{in }\Omega\times(0,T_{max}),\\
  \partial_\nu u=\partial_\nu v=0&\quad\text{on }\partial\Omega\times(0,T_{max}),\\
  u(x,0)=u_0(x),\:v(x,0)=v_0(x)&\quad x\in\Omega,
 \end{cases}
\end{equation}
where $u=u(x,t)$ describes the density of the particels and $v=v(x,t)$ desbribes the density of the chemical.

Even if system \eqref{Poröse Medien mit nichtlinearer Diffusion} is similar to the \textit{classical}
Keller-Segel system of chemotaxis
\begin{equation}\label{klassisches Keller-Segel System}
 \begin{cases}
  u_t=\Delta u-\nabla\cdot(u\nabla v),\\
  v_t=\Delta v-v+u,
 \end{cases}
\end{equation}
which was introduced in \cite{KELLER1970399} and is famous for blow-up phenoma
(see \cite{MR2013508,BELBELTAOWIN2015,MR4071449}),
it is far more complex due to the anisotropic correction term $+\nabla\cdot(u\nabla v)$ in the second equation.
This term arrises whenever the chemical is produced asymmetrically by the colloid \cite{PhysRevLett.115.258301}.
Therefore, the second equation of \eqref{Poröse Medien mit nichtlinearer Diffusion}
is quasilinear even in the case of $m_2=1$ instead of linear as in the classical Keller-Segel
system \eqref{klassisches Keller-Segel System}. The typical question for such systems is whether solutions
exists globally or blow up at finite time.
For \eqref{klassisches Keller-Segel System},
it is known that in two-dimensional domains there are radially symmetric
initial data such that the corresponding solution blows up in finite time with respect to the
function $u$, if $\Vert u_0\Vert_{L^1(\Omega)}>m_c$ \cite{MR1627338} for a certain critical mass $m_c>0$.
Winkler \cite{ScienceDirect:S0021782413000287} proved that in the higher-dimensional setting, for
each $\mu>0$ there are radially symmetric initial data
$(u_0,v_0)\in C^0(\overline{\Omega})\times W^{1,\infty}(\Omega)$ with $\Vert u_0\Vert_{L^1(\Omega)}=\mu$
such that the corresponding solution of \eqref{klassisches Keller-Segel System} blows up at finite time.

In contrast to previously studied versions of \eqref{Poröse Medien mit nichtlinearer Diffusion} with
$m_1=m_2=1$ (see \cite{tao2025suppression,tao2026preventinglpblowuplocal,Zhao_2026} and the discussion below),
we will mainly be interested in system variants with nonlinear (PME-type) diffusion.
For the Keller-Segel system there are several results for the analogous quasilinear case
\begin{equation}\label{Poröse Medien ohne anisotropischen Term}
 \begin{cases}
  u_t=\nabla\cdot(D(u)\nabla u-S(u)\nabla v),\\
  v_t=\Delta v-v+u,
 \end{cases}
\end{equation}
where typical choices are $D(u)=(u+1)^{m-1}$ and $S(u)=u(u+1)^{\sigma-1}$ or $S(u)=u^\sigma$.
Whether all solutions are global or finite time blow-up may occur,
depends on the values of the exponents $m$ and $\sigma$ or the inital mass.
In the setting of $m=\sigma=1$,
we obtain the classical Keller-Segel-system \eqref{klassisches Keller-Segel System}.
Solutions of \eqref{Poröse Medien ohne anisotropischen Term} exist globally if $m$ is sufficiently large
(see \cite{ishida2012global,Cieslak_Stinner_2015,MIMURA20171477} or \cite{Winkler_2019} for $\sigma\leq0$)
or in some settings the initial data
have to be suffciently small \cite{MR4026927}. Finite time blow-up may occur for sufficiently small
values of $m$ (see \cite{Cieslak_Stinner_2015,Hashira_Ishida_Yokota_2018,Cao_Fuest_2025} or
\cite{Laurencot_Mizoguchi_2017} for a critical value in the full space). Similar results are known
for parabolic-elliptic versions in bounded domains (for large $m$ see \cite{Lankeit_2020}, while for
small initial data see \cite{Fuest_Lankeit_Tanaka_2023}) and in the full space
\cite{Sugiyama2007,blanchet2009critical,Liu_Wang_2015}.
Next to explixit choices of $D$ and $S$, we can assume that they satisfy growth conditions like
$\frac{S(u)}{D(u)}\leq C_1(u+\eta)^\gamma$ for constants $C_1,\gamma>0$ and arbitrary $\eta>0$.
Then global and bounded solutions exist if $\gamma<\frac{2}{n}$ and whenever $\eta>0$
\cite{Tao_Winkler_2012,Ishida_Seki_Yokota2014,Tian_Zheng_2016}, while for $\eta=0$ we can find
global weak solutions \cite{Ishida_Seki_Yokota2014}. If $\frac{S(u)}{D(u)}\geq C_2u^\gamma$ for some
$C_2>0$ and $\gamma>\frac{2}{n}$, then there is no global solution \cite{stinner2024critical,Tian_Zheng_2016}.
Diffusion of this kind is what we aim to study for system \eqref{Poröse Medien mit nichtlinearer Diffusion}.
For the chemical we admit nonlinear ($m_2>1$) as well as linear ($m_2=1$) diffusion in addition to its
anisotropic production.
Therefore, we assume the following for system \eqref{Poröse Medien mit nichtlinearer Diffusion}:
\begin{align}\label{Generalvoraussetzungen}
    \begin{split}
       &n\in\N,\:\Omega\subseteq\R^n \text{ to be a smooth bounded domain},\\
       &u_0\in L^2(\Omega),v_0\in C^0(\overline{\Omega}), \text{ to be nonnegative in }\Omega,\\
       &m_1>\frac{n}{2},\:m_1\geq2,\: m_2\geq 1.
       \end{split}
    \end{align}
The case $m_1=m_2=1$ was considered by Tao and Winkler in \cite{tao2025suppression}, where they worked with
\begin{equation}\label{Tao Winkler System}
 \begin{cases}
  u_t=D\Delta u-\nabla\cdot(u\nabla v),\\
  v_t=d\Delta v+\nabla\cdot(u\nabla v)-v+u.
 \end{cases}
\end{equation}
They showed global existence and boundedness of classical solutions in a smooth bounded domain
$\Omega\subseteq\R^n$ with $n\leq 5$,
if $d,D\in[D_0,D_1]$ for postive constants $D_0,D_1$ and if $\vert d-D\vert<\delta$ for a certain $\delta>0$.
The trick in the proof is to introduce a function $w:=u+v$, which solves
\begin{align*}
 w_t=D\Delta w+(d-D)\Delta v-2v+w.
\end{align*}
The advantage is that the cross-diffusion term in the first equation is cancelled, which then allows them to show
global existence and boundedness by using some estimates for the heat equation semigroup
and using boundedness of $\nabla w(\cdot,t)$
and $\nabla v(\cdot,t)$ in $L^q(\Omega)$, $q>n$, for all $t\in(0,T_{max})$.
For arbitrary $D>0$ and $d>0$, Tao and Winkler
\cite{tao2026preventinglpblowuplocal} could show that for nonnegative initial data
$u_0,v_0\in W^{1,\infty}(\Omega)$ with the property
\begin{align*}
 \Vert u_0\Vert_{L^\infty(\Omega)}+\Vert v_0\Vert_{L^\infty(\Omega)}\leq K,
\end{align*}
for a constant $K>0$, one can find nonnegative functions
\begin{align*}
 (u,v)\in \left( L_{loc}^2([0,\infty);\:W^{1,2}(\Omega)) \right)^2,
\end{align*}
and constants $\sigma(K),C(K)>0$ such that $(u,v)$ is a global weak solution of \eqref{Tao Winkler System} and
\begin{align*}
 &\io e^{u^{\sigma(K)}(\cdot,t)}\leq C(K),\quad\text{for a.e. }t>0\\
 &\Vert v(\cdot,t)\Vert_{L^\infty(\Omega)}\leq C(K)\quad\text{for a.e. }t>0.
\end{align*}
The main idea is to use the approximation $\frac{u_\eps}{1+\eps u_\eps}$ in the cross-diffusion term in the
first equation as well as in the signal production in the second equation and consider the function
$w_\eps:=u_\eps+v_\eps$ \cite{tao2026preventinglpblowuplocal}.
In the one-dimensional setting, one obtains global and bounded classical solutions
if $D>0$ is large enough \cite[Theorem 1.1]{Zhao_2026}.
In the case of nonlinear diffusion and without restriction on the spatial dimension,
we obtain the following result for system \eqref{Poröse Medien mit nichtlinearer Diffusion}:
\begin{thm}\label{Thm globale Existenz}
 Suppose \eqref{Generalvoraussetzungen}.
 Then there is a global weak solution $(u,v)$ of
 \eqref{Poröse Medien mit nichtlinearer Diffusion} in the sense of Definition
 \ref{schwache Lösungsdefinition für Grenzsystem}.
 Moreover, there is a constant $C>0$ such that
 \begin{align}\label{Beschränktheit der Lösungen des Grenzsystems im Thm}
  \Vert u(\cdot,t)\Vert_{L^{m_1}(\Omega)}+\Vert v(\cdot,t)\Vert_{L^\infty(\Omega)}\leq C
 \end{align}
for almost all $t\in(0,\infty)$.
\end{thm}
Using the trick with $w=u+v$ does not help anymore here.
The cross-diffusion term would again be cancelled,
but due to the nonlinearity, the diffusion terms cannot be combined into a similar diffusion operator applied to
$w$ only.
While global existence and boundedness was shown directly by applying semigroup estimates to the
$w$-equation in \cite{tao2025suppression} or
using the approximation $\frac{u_\eps}{1+\eps u_\eps}$ in the case of arbitrary, positive
diffusion coefficients \cite{tao2026preventinglpblowuplocal},
respectively, we introduce the following approximating system to obtain global existence of weak solutions
in the sense of Definition \ref{schwache Lösungsdefinition für Grenzsystem}:
\begin{equation}\label{eps>0-System}
 \begin{cases}
  u_{\varepsilon t}=\nabla\cdot((u_\varepsilon+1)^{m_1-1}\nabla u_\varepsilon)
  -\nabla\cdot(u_\varepsilon\frac{\nabla v_\varepsilon}
  {1+\varepsilon\vert\nabla v_\varepsilon\vert})&\quad\text{in }\Omega\times(0,T_{\max}),\\
  v_{\varepsilon t}=\nabla\cdot((v_\varepsilon+1)^{m_2-1}\nabla v_\varepsilon)+\nabla\cdot(u_\varepsilon\nabla v_\varepsilon)
  -v_\varepsilon+u_\varepsilon&\quad\text{in }\Omega\times(0,T_{\max}),\\
  \partial_\nu u_\varepsilon=\partial_\nu v_\varepsilon=0&\quad\text{on }\partial\Omega\times(0,T_{\max}),\\
  u_\varepsilon(x,0)=u_{0,\eps}(x),\:v_\varepsilon(x,0)=v_{0,\eps}(x)&\quad x\in\Omega,
 \end{cases}
\end{equation}
for $\varepsilon>0$ and let
\begin{align}\label{Voraussetzungen an die approximierenden Anfangsdaten}
 \begin{split}
  &\alpha\in(0,1),\\
  &u_{0,\eps}\in C^{2+\alpha}(\overline{\Omega})\quad\text{be nonnegative},\\
  &\partial_\nu u_{0,\eps}=0\quad\text{on }\partial\Omega,\\
  &u_{0,\eps}\to u_0\quad\text{ in }L^2(\Omega)\text{ as }\eps\searrow 0,\\
  &v_{0,\eps}\in C^{2+\alpha}(\overline{\Omega})\quad\text{be nonnegative},\\
  &\partial_\nu v_{0,\eps}=0\quad\text{on }\partial\Omega,\\
  &v_{0,\eps}\to v_0\quad\text{ in }C^0(\overline{\Omega})\text{ as }\eps\searrow 0.
 \end{split}
\end{align}
The reason why we choose a different approximation than in \cite{tao2026preventinglpblowuplocal} is that, as the
computation in Lemma \ref{globale Existenz ueps} shows, it is necessary to weaken the effect of
$\nabla v_\eps$ in the cross-diffusion term. This is more important here than in the classical Keller-Segel
system \eqref{klassisches Keller-Segel System}, because the second equation has less favourable regularization
properties.
Using terms including $\nabla v_\eps$ in the denominator of the cross-diffusive term
also appears in flux limitation models \cite{bellomo2017degenerate, Mizukami2019}.
For the solutions $(u_\eps,v_\eps)$ we will show local existence of classical solutions and
an extensibility criterion. Since we deal with nonlinear diffusion and with a different approximation than
in \cite{tao2026preventinglpblowuplocal}, we will go into more detail in the proof of
Lemma \ref{lokale Existenz klassischer Lösungen für eps>0}.
After that, we can even show global existence and boundedness of solutions, where
the bound for $u_\eps$ depends on $\eps$.
Then our aim is to show that weak solutions of \eqref{Poröse Medien mit nichtlinearer Diffusion}
in the sense
of Definition \ref{schwache Lösungsdefinition für Grenzsystem} are limits of
$(u_\eps,v_\eps)$ as $\eps\searrow 0$.
The main challenge during the computation will be the weak convergence of
$\frac{\nabla v_\eps}{1+\eps\vert\nabla v_\eps\vert}$ to $\nabla v$ in $L^2(\Omega\times(0,T))$.
To be more precise, we have to show $\frac{1}{1+\eps\vert\nabla v_\eps\vert}\to 1$ in $L^2(\Omega\times(0,T))$.
In combination with additional convergence results, which will be proven in
Lemma \ref{schwache Konvergenzen von u v und Gradienten} and Lemma \ref{Normkonvergenzen von ueps und veps},
we obtain the result stated in Theorem \ref{Thm globale Existenz}.

\section{Local existence for the approximating system}
At the beginning, let us state a local existence and uniqueness result for classcial solutions
of system \eqref{eps>0-System}. Since we want to weaken the effect of $\nabla v_\eps$ in the cross-diffusion
term in the first equation of \eqref{eps>0-System}, we use the approximation
$\frac{\nabla v_\eps}{1+\eps\vert\nabla v_\eps\vert}$.
In combination with nonlinear diffusion in both equations,
the situation is more complex than in \cite{tao2025suppression} and \cite{tao2026preventinglpblowuplocal}.
Therefore, we will go more into detail here.
\begin{lem}\label{lokale Existenz klassischer Lösungen für eps>0}
 Let $\Omega\subseteq\R^n$, $n\in\N$, be a smooth bounded domain,
 $\eps>0$, $m_1,m_2\geq 1$.
 Then there exists $T_{max,\eps}\in(0,\infty]$ such that system
 \eqref{eps>0-System} permits a classical solution
 \begin{align*}
  (u_\eps,v_\eps)\in\left( C^0(\overline{\Omega}\times[0,T_{max,\eps}))
 \cap C^{2,1}(\overline{\Omega}\times(0,T_{max,\eps}))\right)^2.
 \end{align*}
 If $T_{max,\eps}<\infty$, then
 \begin{align*}
  \lim_{t\nearrow T_{max}}\left(\Vert u_\eps(\cdot,t)\Vert_{C^{1+\alpha}(\overline{\Omega})}
  +\Vert v_\eps(\cdot,t)\Vert_{C^{2+\alpha}(\overline{\Omega})}\right)=\infty.
 \end{align*}
\end{lem}
\begin{proof}
 \textit{Existence}. Let $\eps>0$, $\beta\in(0,\alpha)$ with $\alpha\in(0,1)$ taken
 from the initial data of \eqref{Voraussetzungen an die approximierenden Anfangsdaten} and let
 \begin{align*}
  &R_1:=2C_1\Vert u_{0,\eps}\Vert_{C^{1+\alpha}(\overline{\Omega})}+1,\\
  &R_2:=2C_2\Vert v_{0,\eps}\Vert_{C^{2+\alpha}(\overline{\Omega})}+C_2R_1+1
 \end{align*}
 where $C_1,C_2>0$ are the constants appearing in the regularity estimates of
 \cite[Theorem 6.49]{Lieberman1996} or
\cite[Theorem 4.28]{Lieberman1996}, respectively,
 and $T\in(0,\min\{C_1^{-\frac{2}{\alpha}},C_2^{-\frac{2}{\alpha}}\})$.
 Then in the Banach space
 $X:=(X_1,X_2)=C^{1+\beta,\frac{1+\beta}{2}}(\overline{\Omega}\times[0,T])
 \times C^{2+\beta,\frac{2+\beta}{2}}(\overline{\Omega}\times[0,T])$,
 we define the set
 \begin{align*}
  S_T:=\{ (\tilde{u}_\eps,\tilde{v}_\eps)\in X\mid \tilde{u}_\eps,
  \tilde{v}_\eps\geq 0\text{ in }\overline{\Omega}\times[0,T),\:
  \Vert\tilde{u}_\eps\Vert_{C^{1+\alpha,\frac{1+\alpha}{2}}}\leq R_1,
  \Vert\tilde{v}_\eps\Vert_{C^{2+\alpha,\frac{2+\alpha}{2}}}\leq R_2\}.
 \end{align*}
 Then $S_T$ is obviously nonempty, convex and closed.
 Moreover, we define an operator $\Phi:S_T\to X,\:(\tilde{u}_\eps,\tilde{v}_\eps)\mapsto(u_\eps,v_\eps)$,
 where $(u_\eps,v_\eps)$ solves
\begin{subequations}\label{Lineariserte Gleichung für ueps in Ex}
\begin{empheq}[left=\empheqlbrace]{align}
&u_{\eps t}
=\nabla\cdot((\tilde{u}_\eps+1)^{m_1-1}\nabla u_\eps)
-\nabla\cdot\left(\frac{u_\eps\nabla\tilde{v}_\eps}
{1+\eps\vert\nabla\tilde{v}_\eps\vert}\right)
&&\text{in }\Omega\times(0,T),\label{Gleichung für ueps Ex}\\
&v_{\eps t}
=\nabla\cdot((\tilde{v}_\eps+1)^{m_2-1}\nabla v_\eps)
+\nabla\cdot(\tilde{u}_\eps\nabla v_\eps)
-v_\eps+\tilde{u}_\eps
&&\text{in }\Omega\times(0,T),\label{Gleichung für veps Ex}\\
&\partial_\nu u_\eps
=\partial_\nu v_\eps=0
&&\text{on }\partial\Omega\times(0,T),\label{Randbedingungen ueps veps Ex}\\
&u_\eps(x,0)
=u_0(x),\: v_\eps(x,0)=v_0(x)
&&\text{in }\Omega\times\{0\}.\label{Anfangsdaten Ex}
\end{empheq}
\end{subequations}
 The existence of solutions $u_\eps$ to \eqref{Gleichung für ueps Ex} and $v_\eps$ to
 \eqref{Gleichung für veps Ex} is guaranteed by \cite[Theorem IV.5.3]{LSU}.
 Our aim is to show that $\Phi$ has a fixed point.
In order to verify structural conditions in parabolic regularity theory results, we now define
\begin{align*}
 &A_u(x,t,z,p):=(\tilde{u}_\eps(x,t)+1)^{m_1-1}\cdot p
 -\frac{z\nabla\tilde{v}_\eps(x,t)}{1+\eps\vert\nabla\tilde{v}_\eps(x,t)\vert},\quad
 B_u(x,t,z,p):=0,
\end{align*}
where $x\in\Omega$, $t\in(0,T)$, $\vert z\vert\leq R_1$, $p\in\R^n$. Then we have
\begin{align*}
 \vert A_u(x,t,z,0)\vert\leq R_1R_2,
\end{align*}
for all $(\tilde{u}_\eps,\tilde{v}_\eps)\in S_T$, $x\in\Omega$, $t\in(0,T)$ and $z\in[-R_1,R_1]$.
By definition of $S_T$ we find $C_3,C_4>0$ such that for all $(\tilde{u}_\eps,\tilde{v}_\eps)\in S_T$,
$x,y\in\Omega$, $t\in(0,T)$, $z,w\in[-R_1,R_1]$ and $p\in\R^n$
\begin{align*}
 &\vert A_u(x,t,z,p)-A_u(y,t,w,p)\vert\\
 =&\vert(\tilde{u}_\eps(x,t)+1)^{m_1-1}\cdot p
 -\frac{z\nabla\tilde{v}_\eps(x,t)}{1+\eps\vert\nabla\tilde{v}_\eps(x,t)\vert}
 -(\tilde{u}_\eps(y,t)+1)^{m_2-1}\cdot p
 +\frac{w\nabla\tilde{v}_\eps(y,t)}{1+\eps\vert\nabla\tilde{v}_\eps(y,t)\vert}\vert\\
 \leq&\vert p\vert\vert(\tilde{u}_\eps(x,t)+1)^{m_1-1}-(\tilde{u}_\eps(y,t)+1)^{m_1-1}\vert
 +\vert w-z\vert\cdot\frac{\vert\nabla\tilde{v}_\eps(y,t)\vert}{1+\eps\vert\nabla\tilde{v}_\eps(y,t)\vert}\\
 &+\vert z\vert\left\vert \frac{\nabla\tilde{v}_\eps(y,t)}{1+\eps\vert\nabla\tilde{v}_\eps(y,t)\vert}-
 \frac{\nabla\tilde{v}_\eps(x,t)}{1+\eps\vert\nabla\tilde{v}_\eps(x,t)\vert}\right\vert\\
 \leq&C_3\vert p\vert\cdot\vert x-y\vert^\alpha+R_2(2R_1)^{1-\alpha}\vert w-z\vert^\alpha+C_4R_1\eps\vert x-y\vert^\alpha\\
 \leq&\max\{ C_3+C_4R_1\eps,R_2(2R_1)^{1-\alpha} \}(1+\vert p\vert)(\vert x-y\vert^\alpha+\vert w-z\vert^\alpha),
\end{align*}
holds true.
For the same reason $C_3,C_4>0$ satisfy
\begin{align*}
 &\vert A_u(x,t,z,p)-A_u(x,s,z,p)\vert\\
 =&\vert(\tilde{u}_\eps(x,t)+1)^{m_1-1}\cdot p-(\tilde{u}_\eps(x,s)+1)^{m_1-1}\cdot p
 -\frac{z\nabla\tilde{v}_\eps(x,t)}{1+\eps\vert\nabla\tilde{v}_\eps(x,t)\vert}
 +\frac{z\nabla\tilde{v}(x,s)}{1+\eps\vert\nabla\tilde{v}(x,s)\vert}\vert\\
 \leq&C_3\vert p\vert\vert t-s\vert^\frac{\alpha}{2}\vert+C_4M\eps\vert t-s\vert^\frac{\alpha}{2}\\
 \leq&\max\{C_3,C_4M\eps\}(1+\vert p\vert)\vert t-s\vert^\frac{\alpha}{2},
\end{align*}
for all $s,t\in(0,T)$.
Since, $B=0$, it holds $\vert B_u(x,t,z,p)\vert\leq(1+\vert p\vert^2)$
for all $(\tilde{u}_\eps,\tilde{v}_\eps)\in S_T$ and
$(x,t,z,p)\in\Omega\times(0,T)\times\R\times\R^n$.
Now let us define $a^{ij}:=\frac{\partial A_u}{\partial p}=(\tilde{u}_\eps+1)^{m_1-1}I$. Then,
 $\vert a^{ij}(x,t,z,p)-a^{ij}(x,t,z,q)\vert=0$
for all $p,q\in\R^n$. Moreover,
\begin{align*}
 a^{ij}(x,t,z,p)\xi_i\xi_j\geq\vert\xi\vert^2,\quad\vert a^{ij}(x,t,z,p)\vert\leq(R_1+1)^{m_1-1}
\end{align*}
for all $(\tilde{u}_\eps,\tilde{v}_\eps)\in S_T$, $(x,t)\in\Omega\times(0,T)$, $z\in[-R_1,R_1]$ and
$p,\xi\in\R^n$.
By definition of $S_T$, we find $C_5,C_6>0$ such that
for all $(\tilde{u}_\eps,\tilde{v}_\eps)\in S_T$,
$x\in\Omega$, $s,t\in[0,T]$, $\vert z\vert\leq R_1$ and $\vert p\vert\leq N$
\begin{align*}
 &\vert A_u(x,t,z,p)-A_u(x,s,z,p)\vert
 \leq C_5\vert p\vert\vert t-s\vert^\frac{\alpha}{2}+C_6R_1\eps\vert t-s\vert^\frac{\alpha}{2}\\
 =&(C_5\vert p\vert+C_6R_1\eps)\vert t-s\vert^\frac{\alpha}{2}
 \leq(C_5N+C_6R_1\eps)\vert t-s\vert^\frac{\alpha}{2}.
\end{align*}
Hence, by \cite[Theorem 1.1, Theorem 1.2]{MR932759} there is $C_7>0$
such that for all $(\tilde{u}_\eps,\tilde{v}_\eps)\in S_T$
\begin{align}\label{ueps beschränkt in C1+alpha}
\Vert u_\eps\Vert_{C^{1+\alpha,\frac{1+\alpha}{2}}(\overline{\Omega}\times[0,T])}\leq C_7.
\end{align}
Since \eqref{Gleichung für ueps Ex} is linear in $u_\eps$,
we can apply the comparison theorem of \cite[Theorem B.1]{Lankeit_Lankeit_2019} to conclude
that $u_\eps\geq 0$ in $\Omega\times(0,T)$. Moreover, by \cite[Theorem 6.49]{Lieberman1996} we find $C_1>0$ such that
\begin{align*}
 \Vert u_\eps\Vert_{C^{1+\alpha,\frac{1+\alpha}{2}}(\overline{\Omega})}\leq
 &C_1(\Vert u_\eps\Vert_{C^0(\overline{\Omega}\times[0,T])}+\Vert u_{0,\eps}\Vert_{C^{1+\alpha}(\overline{\Omega})})\\
 \leq&C_1(\sup_{t\in(0,T)}\Vert u_\eps(\cdot,t)\Vert_{C^0(\overline{\Omega}\times[0,T])}
 +\Vert u_{0,\eps}\Vert_{C^\alpha(\overline{\Omega})}
 +\Vert u_{0,\eps}\Vert_{C^{1+\alpha}(\overline{\Omega})})\\
 \leq&C_1(t^\frac{\alpha}{2}+\Vert u_{0,\eps}\Vert_{C^0(\overline{\Omega})}
 +\Vert u_{0,\eps}\Vert_{C^{1+\alpha}(\overline{\Omega})})\\
 \leq&1+2C_1\Vert u_{0,\eps}\Vert_{C^{1+\alpha}(\overline{\Omega})}\\
 =& R_1,
\end{align*}
since $t<T<C_1^{-\frac{2}{\alpha}}$.
Since all coefficients in \eqref{Gleichung für ueps Ex} are Hölder continuous, it holds
$u_\eps\in C^{2+\theta_2,1+\frac{\theta_2}{2}}(\overline{\Omega}\times(0,T))$
for some $\theta_2\in(0,1)$ by \cite[Theorem 4.28]{Lieberman1996}.\\
For the second solution component, by \cite[Theorem 4.28]{Lieberman1996} there is $C_2>0$ such that
\begin{align}\label{v beschränkt in C2+alpha}
\begin{split}
 \Vert v_\eps\Vert_{C^{2+\alpha,\frac{2+\alpha}{2}}(\overline{\Omega}\times[0,T])}
 \leq&C_2(\Vert v_\eps\Vert_{C^0(\overline{\Omega}\times[0,T])}
 +\Vert v_{0,\eps}\Vert_{C^{2+\alpha}(\overline{\Omega})}
 +\Vert\tilde{u}_\eps\Vert_{C^{\alpha,\frac{\alpha}{2}}(\overline{\Omega}\times[0,T])})\\
 \leq&C_2(\sup_{t\in(0,T)}\Vert v_\eps(\cdot,t)\Vert_{C^0(\overline{\Omega})}
 +\Vert v_{0,\eps}\Vert_{C^{2+\alpha}(\overline{\Omega})}+R_1)\\
 \leq& C_2\sup_{t\in(0,T)}t^{\frac{\alpha}{2}}+2C_2\Vert v_{0,\eps}\Vert_{C^{2+\alpha}(\overline{\Omega})}
 +C_2R_1\\
 \leq&C_2T^{\frac{\alpha}{2}}+2C_2\Vert v_{0,\eps}\Vert_{C^{2+\alpha}(\overline{\Omega})}
 +C_2R_1\\
 \leq&1+2C_2\Vert v_{0,\eps}\Vert_{C^{2+\alpha}(\overline{\Omega})}+C_2R_1\\
 =& R_2,
\end{split}
\end{align}
since $T<C_2^{-\frac{2}{\alpha}}$. Hence, $v_\eps\in X_2$. Moreover,
by the comparison theorem of \cite[Theorem B.1]{Lankeit_Lankeit_2019} again,
we can conclude $v_\eps\geq 0$ in $\Omega\times(0,T)$. Hence, $(u_\eps,v_\eps)\in S_T$.
Hence, we have a pair $(u_\eps,v_\eps)$ of
classical solutions to \eqref{Lineariserte Gleichung für ueps in Ex}.\\
Our aim is to show that $\Phi$ has a fixed point by Schauder's fixed point theorem.
Since $\Phi(S_T)\subseteq S_T$, it remains to show that $\Phi$ is continuous and compact.
For continuity of $\Phi$,
let $(\tilde{u}_{\eps,1},\tilde{v}_{\eps,1}),(\tilde{u}_{\eps,2},\tilde{v}_{\eps,2})\in S_T$.
Moreover, let
$\tilde{U}_\eps:=\tilde{u}_{\eps,1}-\tilde{u}_{\eps,2}$, $\tilde{V}_\eps:=\tilde{v}_{\eps,1}-\tilde{v}_{\eps,2}$
and $(U_\eps,V_\eps)\in S_T$ such that $\Phi(\tilde{U}_\eps,\tilde{V}_\eps)=(U_\eps,V_\eps)$.
Then $(U_\eps,V_\eps)$ solves
\begin{equation*}
 \begin{cases}
  U_{\eps t}=\nabla\cdot (a_U\nabla U_\eps+b_U U_\eps)+\nabla\cdot F_U,\\
  V_{\eps t}= A_V+F_V,\\
  U_\eps(\cdot,0)=0, V_\eps(\cdot,0)=0.
 \end{cases}
\end{equation*}
where
\begin{align*}
 a_U:=&(\tilde{u}_\eps+1)^{m_1-1},\quad
 b_U:=-\frac{\nabla\tilde{v}_{\eps,2}}{1+\eps\vert\nabla\tilde{v}_{\eps,2}\vert},\\
 F_U:=&((\tilde{u}_{\eps,1}+1)^{m_1-1}-(\tilde{u}_{\eps,2}+1)^{m_1-1})\nabla u_{\eps,2}
  -\frac{u_{\eps,1}\nabla \tilde{V}_\eps}{1+\eps\vert\nabla\tilde{v}_{\eps,1}\vert}
  -(u_{\eps,1}\nabla\tilde{v}_{\eps,2})\cdot\left( \frac{1}{1+\eps\vert\nabla \tilde{v}_{\eps,1}\vert}
  -\frac{1}{1+\eps\vert\nabla \tilde{v}_{\eps,2}\vert} \right),\\\\
 A_V:=&\nabla\cdot((\tilde{v}_\eps+1)^{m_1-1}\nabla V_\eps+\tilde{u}_{\eps,2}\nabla V_\eps)-V_\eps,\\\\
 F_V:=&((\tilde{v}_{\eps,1}+1)^{m_2-1}-(\tilde{v}_{\eps,2}+1)^{m_2-1})\nabla v_{\eps,2}
 +\tilde{U}_\eps\nabla v_{\eps,1}+\tilde{U}_\eps.
\end{align*}
Then by \cite[Theorem 6.49]{Lieberman1996} with $LU:=-U_{\eps t}+\nabla\cdot(A_U)$, we find $C_{8}>0$
and due to
$\tilde{U}_\eps,\tilde{V}_\eps\in C^{1+\beta,\frac{1+\beta}{2}}(\overline{\Omega}\times[0,T])$ we find
constants $C_{9},C_{10},C_{11}>0$ such that
\begin{align*}
 &\Vert U_\eps\Vert_{C^{1+\beta,\frac{1+\beta}{2}}(\overline{\Omega}\times[0,T])}
 \leq C_{8}\Vert F_U\Vert_{C^{\beta,\frac{\beta}{2}}(\overline{\Omega}\times[0,T])}\\
 \leq& C_{9}\Vert \tilde{U}_\eps\Vert_{C^{\beta,\frac{\beta}{2}}(\overline{\Omega}\times[0,T])}
 +(C_{10}+\eps C_{11})\Vert\tilde{V}_\eps\Vert_{C^{1+\beta,\frac{1+\beta}{2}}(\overline{\Omega}\times[0,T])}\\
 \leq&C_{9}\Vert\tilde{U}_\eps\Vert_{C^{1+\beta,\frac{1+\beta}{2}}(\overline{\Omega}\times[0,T])}
 +(C_{10}+\eps C_{11})\Vert\tilde{V}_\eps\Vert_{C^{2+\beta,\frac{2+\beta}{2}}(\overline{\Omega}\times[0,T])},
\end{align*}
and by \cite[Theorem IV.5.3]{LSU} there is $C_{12}>0$ and due to
$\tilde{U}_\eps,\tilde{V}_\eps\in C^{1+\beta,\frac{1+\beta}{2}}(\overline{\Omega}\times[0,T])$, there are
$C_{13},C_{14}>0$ such that
\begin{align*}
 &\Vert V_\eps\Vert_{C^{2+\beta,\frac{2+\beta}{2}}(\overline{\Omega}\times[0,T])}
 \leq C_{12}\Vert F_V\Vert_{C^{\beta,\frac{\beta}{2}}(\overline{\Omega}\times[0,T])}\\
 \leq&C_{13}\Vert\tilde{V}_\eps\Vert_{C^{\beta,\frac{\beta}{2}}(\overline{\Omega}\times[0,T])}
 +C_{14}\Vert\tilde{U}_\eps\Vert_{C^{\beta,\frac{\beta}{2}}(\overline{\Omega}\times[0,T])}\\
 \leq&C_{13}\Vert\tilde{V}_\eps\Vert_{C^{2+\beta,\frac{2+\beta}{2}}(\overline{\Omega}\times[0,T])}
 +C_{14}\Vert\tilde{U}_\eps\Vert_{C^{1+\beta,\frac{1+\beta}{2}}(\overline{\Omega}\times[0,T])}
 \end{align*}
Thus, we find $C_{15}, C_{16}(\eps)>0$ such that
\begin{align*}
 &\Vert\Phi((\tilde{u}_{\eps,1},\tilde{v}_{\eps,1}))-\Phi((\tilde{u}_{\eps,2},\tilde{v}_{\eps,2}))\Vert^2_X
 =\Vert(U_\eps,V_\eps)\Vert_X^2\\
 \leq& C_{15}\Vert \tilde{U}_\eps\Vert^2_{C^{1+\beta,\frac{1+\beta}{2}}(\overline{\Omega}\times[0,T])}+
  C_{16}(\eps)\Vert\tilde{V}_\eps\Vert^2_{C^{2+\beta,\frac{2+\beta}{2}}(\overline{\Omega}\times[0,T])}\\
  \leq&\max\{C_{15},C_{16}(\eps)\}\Vert(\tilde{U}_\eps,\tilde{V}_\eps)\Vert^2_X.
\end{align*}
Hence, $\Phi$ is continuous. Due to \eqref{ueps beschränkt in C1+alpha}, \eqref{v beschränkt in C2+alpha}
and  $\alpha>\beta$, compactness of $\Phi$ follows by the Arzel\`{a}-Ascoli theorem.
Hence, by Schauder's fixed point theorem there is a pair $(u_\eps,v_\eps)\in S_T$
such that $\Phi((u_\eps,v_\eps))=(u_\eps,v_\eps)$.\\
\textit{Extensibility criterion.} Let $Z:=\{ T\in(0,\infty)\mid \exists
\text{ classcial solution }(u_{\eps},v_\eps)\in
 C^{2,1}(\overline{\Omega}\times(0,T))\cap C^0(\overline{\Omega}\times[0,T))\}$.
Assume $\tilde{T}:=T_{max,\eps}=\sup Z<\infty$
and that there is a constant $C_{17}>0$ such that
\begin{align*}
 \limsup_{t\nearrow T_{max,\eps}}(\Vert u_\eps(\cdot,t)\Vert_{C^{1+\alpha}(\overline{\Omega})}
 +\Vert v_\eps(\cdot,t)\Vert_{C^{2+\alpha}(\overline{\Omega})})\leq C_{17}.
\end{align*}
Due to the local existence, we find $\tau=\tau(C_{17})>0$ such that
for all $(u_{0,\eps},v_{0,\eps})\in\left( C^{2+\alpha}(\overline{\Omega})\right)^2$ that have
the property
\begin{align*}
 \Vert u_{0,\eps}\Vert_{C^{2+\alpha}(\overline{\Omega})}
 +\Vert v_{0,\eps}\Vert_{C^{2+\alpha}(\overline{\Omega})}\leq C_{17},
\end{align*}
there exists a classical solution $(u_\eps,v_\eps)$ of \eqref{eps>0-System} in $\Omega\times[0,\tau)$.
Furthermore, we find a sequence $(t_k)_{k\in\N}\subseteq(0,\tilde{T})$
such that $t_k\nearrow\tilde{T}$ as $k\to\infty$ and
\begin{align*}
 \Vert u_\eps(\cdot,t_k)\Vert_{C^{1+\alpha}(\overline{\Omega})}
 +\Vert v_\eps(\cdot,t_k)\Vert_{C^{2+\alpha}(\overline{\Omega})}
 \leq C_{17},
\end{align*}
for all $k\in\N$. We define a sequence of initial data by
\begin{align*}
 u_{0_k,\eps}(x):=u_\eps(x,t_k),\quad x\in\Omega.
\end{align*}
Due to $t_k\nearrow\tilde{T}$ there is $\tilde{k}\in\N$ such that
\begin{align*}
 \tilde{T}-t_k<\frac{\tau}{2},
\end{align*}
for all $k\geq\tilde{k}$. If we set $u_{0_{\tilde{k}},\eps}(x):=u_\eps(x,\tilde{T}-\frac{\tau}{2})$,
then the function
\begin{align*}
 \overline{u}_\eps(t):=\left\{\begin{array}{cc}
            u_\eps(t),  &t<\tilde{T}-\frac{\tau}{2}  \\
            u_{\tilde{k},\eps}(t),& t\in[\tilde{T}-\frac{\tau}{2},\tilde{T}+\frac{\tau}{2})
        \end{array} \right.
\end{align*}
is a weak solution to the first equation of \eqref{Poröse Medien mit nichtlinearer Diffusion}.
Due to the regularity results from above (cf. \cite[Theorem 4.28]{Lieberman1996}),
it is also a classcial solution.
Hence, $\tilde{T}+\frac{\tau}{2}\in Z$.
Since $\tilde{T}=T_{max,\eps}$, it follows $\sup Z=\tilde{T}>\tilde{T}+\frac{\tau}{2}\geq
\tilde{T}=\sup Z$. Hence,
\begin{align*}
 \lim_{t\nearrow T_{max}}(\Vert u_\eps(\cdot,t)\Vert_{C^{1+\alpha}(\overline{\Omega})}
 +\Vert v_\eps(\cdot,t)\Vert_{C^{2+\alpha}(\overline{\Omega})})=\infty.
\end{align*}
\end{proof}

We can improve the extensibility criterion due to well-known parabolic regularity theory \cite{MR932759}.

\begin{lem}\label{verbessertes Fortsetzungskriterium}
 Let $\eps>0$ and $(u_\eps,v_\eps)$ be a classcial solution of \eqref{eps>0-System}
 on $\Omega\times(0,T_{max,\eps})$.
 If $T_{max,\eps}<\infty$, then
 \begin{align}\label{Fortsetzungskriterium in Linfty}
  \lim_{t\nearrow T_{max,\eps}}\Vert u_\eps(\cdot,t)\Vert_{L^\infty(\Omega)}=\infty.
 \end{align}
\end{lem}
\begin{proof}
Let $T\in(0,T_{max,\eps}]\cap(0,\infty)$ and use
\eqref{ueps beschränkt in C1+alpha} from the
proof of Lemma \ref{lokale Existenz klassischer Lösungen für eps>0}.
By \cite[Theorem 1.1, Theorem 1.2]{MR932759} we obtain a constant
$C_1=C_1(\Omega,T,m_1,R_1,\Vert u_\eps\Vert_{L^\infty(\Omega\times(0,T))})>0$ such that
 \begin{align}\label{u C1+alpha gegen Linfty abschätzen}
  \Vert u_\eps\Vert_{C^{1+\alpha,\frac{1+\alpha}{2}}(\overline{\Omega}\times[0,T])}\leq C_1.
 \end{align}
 By \eqref{v beschränkt in C2+alpha} we find $C_2>0$ and by \eqref{u C1+alpha gegen Linfty abschätzen} we use $C_1>0$
 to estimate
 \begin{align*}
  \Vert v_{\eps}\Vert_{C^{2+\alpha,\frac{2+\alpha}{2}}(\overline{\Omega}\times[0,T])}
  \leq &C_2(T+\Vert v_{0,\eps}\Vert_{C^{2+\alpha}(\overline{\Omega})}
  +\Vert u_\eps\Vert_{C^{\alpha,\frac{\alpha}{2}}(\overline{\Omega}\times[0,T])})\\
  \leq&C_2(T+\Vert v_{0,\eps}\Vert_{C^{2+\alpha}(\overline{\Omega})}
  +C_1)
 \end{align*}
 Hence, boundedness of $\Vert u_\eps\Vert_{L^\infty(\Omega\times(0,T))}$
  implies boundedness of
 $\Vert u_\eps\Vert_{C^{1+\alpha,\frac{1+\alpha}{2}}(\overline{\Omega}\times[0,T])}$ and
 $\Vert v_\eps\Vert_{C^{2+\alpha,\frac{2+\alpha}{2}}(\overline{\Omega}\times[0,T])}$.
\end{proof}

\begin{rem}
 In \cite{tao2025suppression} and \cite{tao2026preventinglpblowuplocal}, local existence and an extensibility
 criterion
 are proven by Amann's theory \cite{amann1993nonhomogeneous} instead of applying
 \cite{Lieberman1996} and \cite{LSU}.
 The extensibility criteria are formulated in $W^{1,q}(\Omega)$
 for all $q>n$ \cite[Lemma 3.1]{tao2025suppression}. In the setting of linear diffusion, boundedness in
 $W^{1,q}(\Omega)$ is obtained in \cite[Lemma 3.5, Lemma 4.1, Lemma 5.6]{tao2025suppression}.
 As in Lemma \ref{verbessertes Fortsetzungskriterium}, they could be reduced to the $L^\infty$-norm.
\end{rem}

\section{Bounds for u and v}
As our first result, we will show that the total mass of $u_\eps$ is constant and the total mass of
$v_\eps$ is bounded uniformly in $t$.
\begin{lem}\label{L1 Norm von ueps konstant und L1 Norm von veps beschränkt}
 Suppose \eqref{Generalvoraussetzungen} and \eqref{Voraussetzungen an die approximierenden Anfangsdaten}. Then
 there is $C>0$ such that for each $\eps>0$ any classical solution $(u_\eps,v_\eps)$
 of \eqref{eps>0-System} satisfies
 \begin{align}\label{L1 Norm von u konstant}
  \io u_\eps(\cdot,t)=\io u_{0,\eps}\quad\text{for all } t\in(0,T_{max,\eps}),
 \end{align}
and
\begin{align}\label{L1 Norm von v beschränkt}
 \io v_\eps(\cdot,t)\leq C\quad\text{for all } t\in(0,T_{max,\eps}).
\end{align}
\end{lem}
\begin{proof}
 Integrating the first equation of \eqref{Poröse Medien mit nichtlinearer Diffusion} over $\Omega$ and applying
 the Gaussian divergence theorem yields
 \begin{align*}
  &\frac{d}{dt}\io u_\eps=\io u_{\eps t}=\io\nabla\cdot((u_\eps+1)^{m_1-1}\nabla u_\eps)
  -\io\nabla\cdot\left(u_\eps\frac{\nabla v_\eps}{1+\eps\vert\nabla v_\eps\vert}\right)\\
  =&\int_{\partial\Omega}(u_\eps+1)^{m_1-1}\nabla u_\eps\cdot\nu
  -\int_{\partial\Omega}u_\eps\frac{\nabla v_\eps\cdot\nu}{1+\eps\vert\nabla v_\eps\vert}=0,
 \end{align*}
 on account of the homogenuous boundary conditions.
The estimate \eqref{L1 Norm von v beschränkt} can be obtained
by integration of the second equation of \eqref{eps>0-System}
over $\Omega$
\begin{align*}
 &\frac{d}{dt}\io v_\eps=\io v_{\eps t}=\io\nabla\cdot((v_\eps+1)^{m_2-1}\nabla v_\eps)
 +\io\nabla\cdot(u_\eps\nabla v_\eps)-\io v_\eps+\io u_\eps\\
 &=\int_{\partial\Omega}(v_\eps+1)^{m_2-1}\nabla v_\eps\cdot\nu
 +\int_{\partial\Omega}u_\eps\nabla v_\eps\cdot\nu-\io v_\eps+\io u_\eps=-\io v_\eps+\io u_\eps,
\end{align*}
where the boundary terms vanish again due to the homogenuous boundary conditions.
By \eqref{Voraussetzungen an die approximierenden Anfangsdaten},
we can find a constant $C_1>0$ such that (in accordance with \eqref{L1 Norm von u konstant})
\begin{align*}
 \frac{d}{dt}\io v_\eps+\io v_\eps=\io u_{0,\eps}\leq C_1,\quad\text{on }(0,T_{max,\eps})\text{ for each }\eps>0.
\end{align*}
Hence, $y(t):=\io v_\eps(\cdot,t)$ satisfies the ODI $y'(t)+y(t)\leq C_1$,
which implies $y(t)\leq\max\{\io v_{0,\eps},C_1\}$
for all $t\in(0,T_{max,\eps})$,
and hence boundedness of $y$, by \eqref{Voraussetzungen an die approximierenden Anfangsdaten}.
\end{proof}
Our next goal is to show global existence of classical solutions to system \eqref{eps>0-System}. As a first step,
we derive some bounds for $u_\eps$ and $v_\eps$ in $L^p(\Omega)$ that are uniform in time.
Therefore, let us recall an useful Ehrling-type estimate.
\begin{lem}\label{GN L2 Lp mit p in (0,2)}\cite[Lemma 3.2]{tao2025suppression}
Let $p\in(0,2)$ and $\eta>0$. Then there exists $C(p,\eta)>0$ such that
\begin{align*}
 \Vert\varphi\Vert_{L^2(\Omega)}^2\leq\eta\Vert\nabla\varphi\Vert_{L^2(\Omega)}^2+C(p,\eta)\Vert\varphi\Vert_{L^p(\Omega)}^2\quad
 \forall\varphi\in W^{1,2}(\Omega).
\end{align*}
\end{lem}

\begin{lem}\label{ueps und veps in Lm1 beschränkt}
 Suppose \eqref{Generalvoraussetzungen} and \eqref{Voraussetzungen an die approximierenden Anfangsdaten}.
 Then there is $C>0$ such that for all $\eps\in(0,1)$ the corresponding
 classical solutions $(u_\eps,v_\eps)$ of \eqref{eps>0-System} satisfy
 \begin{align*}
  \Vert u_{\eps}(\cdot,t)\Vert_{L^{m_1}(\Omega)}+\Vert (v_{\eps}+1)(\cdot,t)\Vert_{L^{m_1}(\Omega)}
  \leq C\quad\forall t\in(0,T_{max,\eps}).
 \end{align*}
 \end{lem}
 \begin{proof}
 We define $y_\eps(t):=\frac{1}{m_1}\left( \io u_\eps^{m_1}+\io(v_\eps+1)^{m_1} \right)$.
 Then we use both equations of
 \eqref{eps>0-System} and integrate by parts to obtain
  \begin{align*}
   &y_\eps'(t)=\frac{1}{m_1}\frac{d}{dt}\left( \io u_{\eps}^{m_1}+\io(v_{\eps}+1)^{m_1} \right)\\
   =&-(m_1-1)\io u_{\eps}^{m_1-2}(u_{\eps}+1)^{m_1-1}\vert\nabla u_{\eps}\vert^2
   +(m_1-1)\io u_{\eps}^{m_1-1}\nabla u_{\eps}\cdot\frac{\nabla v_{\eps}}{1+\eps\vert\nabla v_{\eps}\vert}\\
   &-(m_1-1)\io(v_{\eps}+1)^{m_1+m_2-3}\vert\nabla v_{\eps}\vert^2
   -(m_1-1)\io(v_{\eps}+1)^{m_1-2}u_{\eps}\vert\nabla v_{\eps}\vert^2\\
   &-\io(v_{\eps}+1)^{m_1-1}v_{\eps}+\io(v_{\eps}+1)^{m_1-1}u_{\eps}.
   \end{align*}
Then for the second integral we apply Young's inequality, which yields
\begin{align*}
 (m_1-1)\io u_{\eps}^{m_1-1}\nabla u_{\eps}\cdot\frac{\nabla v_{\eps}}{1+\eps\vert\nabla v_{\eps}\vert}
 \leq&(m_1-1)\io u_{\eps}^{\frac{m_1-2}{2}}u_{\eps}^\frac{1}{2}(u_{\eps}+1)^{\frac{m_1-1}{2}}
   \vert\nabla u_{\eps}\vert\vert\nabla v_{\eps}\vert\\
   \leq&\frac{m_1-1}{2}\io u_{\eps}^{m_1-2}(u_{\eps}+1)^{m_1-1}\vert\nabla u_{\eps}\vert^2
   +\frac{m_1-1}{2}\io u_{\eps}\vert\nabla v_{\eps}\vert^2.
\end{align*}
 Since $m_1\geq 2$ and $(v_{\eps}+1)\geq 1$ for all $\eps>0$, we have
\begin{align*}
 &\frac{m_1-1}{2}\io u_{\eps}\vert\nabla v_{\eps}\vert^2
 -(m_1-1)\io(v_{\eps}+1)^{m_1-2}u_{\eps}\vert\nabla v_{\eps}\vert^2
 \leq-\frac{m_1-1}{2}\io u_{\eps}\vert\nabla v_{\eps}\vert^2.
\end{align*}
In addition, due to $m_1+m_2-3\geq m_1-2$, we have
\begin{align*}
 -(m_1-1)\io(v_{\eps}+1)^{m_1+m_2-3}\vert\nabla v_{\eps}\vert^2
 \leq-(m_1-1)\io(v_{\eps}+1)^{m_1-2}\vert\nabla v_{\eps}\vert^2.
\end{align*}
Hence, after adding $m_1 y_\eps(t)$ on both sides we have
\begin{align*}
 &y_\eps'(t)+m_1y_\eps(t)+\frac{m_1-1}{2}\io u_{\eps}^{m_1-2}(u_{\eps}+1)^{m_1-1}\vert\nabla u_{\eps}\vert^2
 +\frac{m_1-1}{2}\io u_{\eps}\vert\nabla v_{\eps}\vert^2\\
 &+(m_1-1)\io(v_{\eps}+1)^{m_1-2}\vert\nabla v_{\eps}\vert^2+\io(v_{\eps}+1)^{m_1-1}v_{\eps}\\
 \leq&\frac{2m_1-1}{m_1}\io(v_{\eps}+1)^{m_1}+\frac{m_1+1}{m_1}\io u_{\eps}^{m_1}.
\end{align*}
An application of Lemma \ref{GN L2 Lp mit p in (0,2)} with the choice of $\eta_1=\frac{4(m_1-1)}{(2m_1-1)m_1}$ yields
$C_1=C_1(p,\eta_1)$ such that
\begin{align*}
 \frac{2m_1-1}{m_1}\io(v_{\eps}+1)^{m_1}&=\frac{2m_1-1}{m_1}\Vert(v_{\eps}+1)^\frac{m_1}{2}\Vert_{L^2(\Omega)}^2\\
 &\leq\frac{4(m_1-1)}{m_1^2}\Vert\nabla(v_{\eps}+1)^{\frac{m_1}{2}}\Vert_{L^2(\Omega)}^2
 +C_1\Vert(v_{\eps}+1)^\frac{m_1}{2}\Vert_{L^\frac{2}{m_1}(\Omega)}^2\\
 &=\frac{4(m_1-1)}{m_1^2}\Vert\nabla(v_{\eps}+1)^\frac{m_1}{2}\Vert_{L^2(\Omega)}^2
 +C_1\left( \io(v_{\eps}+1) \right)^{m_1}\\
 &\leq\frac{4(m_1-1)}{m_1^2}\Vert\nabla(v_{\eps}+1)^\frac{m_1}{2}\Vert_{L^2(\Omega)}^2+C_2,
\end{align*}
where the constant $C_2>0$ can be obtained by an application of
Lemma \ref{L1 Norm von ueps konstant und L1 Norm von veps beschränkt}.
The second term can be estimated the same way. Another application of Lemma \ref{GN L2 Lp mit p in (0,2)} with
the choice of $\eta_2=\frac{2(m_1-1)}{m_1(m_1+1)}$ yields $C_3=C_3(p,\eta_2)>0$ such that
\begin{align*}
 \frac{m_1+1}{m_1}\io u_{\eps}^{m_1}=&\frac{m_1+1}{m_1}\Vert u_{\eps}^\frac{m_1}{2}\Vert_{L^2(\Omega)}^2\\
 \leq&\frac{2(m_1-1)}{m_1^2}\Vert\nabla u_{\eps}^\frac{m_1}{2}\Vert_{L^2(\Omega)}^2
 +C_3\Vert u_{\eps}^\frac{m_1}{2}\Vert_{L^\frac{2}{m_1}(\Omega)}^2\\
 =&\frac{2(m_1-1)}{m_1^2}\Vert\nabla u_{\eps}^\frac{m_1}{2}\Vert_{L^2(\Omega)}^2
 +C_3\left( \io u_{\eps} \right)^{m_1}\\
 \leq& \frac{2(m_1-1)}{m_1^2}\Vert\nabla u_{\eps}^\frac{m_1}{2}\Vert_{L^2(\Omega)}^2+C_4,
\end{align*}
where the constant $C_4>0$ can be obtained by another application of
Lemma \ref{L1 Norm von ueps konstant und L1 Norm von veps beschränkt}. Since $\frac{4}{m_1^2}\leq 1$ and
$\frac{2}{m_1^2}\leq\frac{1}{2}$, the function $y_\eps$
satisfies $y_\eps'(t)+y_\eps(t)\leq C_3+C_4$ for all $\eps>0$ and
$t\in(0,T_{\max_\eps})$, which implies the existence of $C_5>0$ such that $y_\eps(t)\leq C_5$
for all $\eps>0$ and $t\in(0,T_{\max,\eps})$.
 \end{proof}
 Moreover, we can show some bounds for $\nabla v_\eps$ and $\nabla u_\eps$ in $L^2(\Omega\times(0,T))$ as long as
 $T<T_{max,\eps}$.
\begin{lem}\label{L2 Raum Zeit Beschränktheit von nabla v}
 Suppose \eqref{Generalvoraussetzungen} and \eqref{Voraussetzungen an die approximierenden Anfangsdaten}.
 Then there is $C>0$ such that for all $\eps>0$ and $T\in(0,T_{max,\eps}]\cap(0,\infty)$,
 the corresponding classical solutions $(u_\eps,v_\eps)$ of \eqref{eps>0-System} satisfy
 \begin{align*}
  \io v_{\eps}^2(\cdot,T)+\int_0^T\io\vert\nabla v_{\eps}\vert^2
  +\int_0^T\io(v_{\eps}+1)^{m_2-1}\vert\nabla v_{\eps}\vert^2
  +\int_0^T\io u_{\eps}\vert\nabla v_{\eps}\vert^2
  +\int_0^T\io v_{\eps}^2\leq C(T+1).
 \end{align*}
\end{lem}

\begin{proof}
 We test the second equation of \eqref{eps>0-System} by $v_{\eps}$. Then integration by parts and
 Young's inequality yield
 \begin{align*}
  \frac{1}{2}\frac{d}{dt}\io v_{\eps}^2=&-\io(v_{\eps}+1)^{m_2-1}\vert\nabla v_{\eps}\vert^2
  -\io u_{\eps}\vert\nabla v_{\eps}\vert^2-\io v_{\eps}^2+\io u_{\eps}v_{\eps}\\
  \leq&-\io(v_{\eps}+1)^{m_2-1}\vert\nabla v_{\eps}\vert^2-\io u_{\eps}\vert\nabla v_{\eps}\vert^2
  -\frac{1}{2}\io v_{\eps}^2+\frac{1}{2}\io u_{\eps}^2 \quad\text{in }(0,T_{max,\eps})\text{ for each }\eps>0.
 \end{align*}
Hence, by Lemma \ref{ueps und veps in Lm1 beschränkt} and \eqref{Voraussetzungen an die approximierenden Anfangsdaten},
there is $C>0$ such that for all $\eps>0$ and $T\in(0,T_{max})$
\begin{align*}
 &\frac{1}{2}\io v_{\eps}^2(\cdot,T)+\frac{1}{2}\int_0^T\io\vert\nabla v_{\eps}\vert^2
 +\frac{1}{2}\int_0^T\io(v_{\eps}+1)^{m_2-1}\vert\nabla v_{\eps}\vert^2
 +\int_0^T\io u_{\eps}\vert\nabla v_{\eps}\vert^2+\frac{1}{2}\int_0^T\io v_{\eps}^2\\
 \leq&\frac{1}{2}\io v_{0,\eps}^2+\frac{1}{2}\int_0^T\io u_{\eps}^2\leq C(T+1).
\end{align*}
\end{proof}

\begin{lem}\label{L2 Raum Zeit Beschränktheit von nabla u}
 Suppose \eqref{Generalvoraussetzungen} and \eqref{Voraussetzungen an die approximierenden Anfangsdaten}.
 Then there is $C>0$ such that for all $\eps>0$ and $T\in(0,T_{max,\eps}]\cap(0,\infty)$,
 the corresponding classical solutions $(u_\eps,v_\eps)$ of \eqref{eps>0-System} satisfy
 \begin{align*}
  &\io u_{\eps}^2(\cdot,T)+\int_0^T\io \vert\nabla u_{\eps}\vert^2
  +\int_0^T\io(u_{\eps}+1)^{m_1-1}\vert\nabla u_{\eps}\vert^2
  +\int_0^T\io u_{\eps}\vert\nabla u_{\eps}\vert^2\leq C(T+1).
 \end{align*}
\end{lem}

\begin{proof}
 We test the first equation of \eqref{eps>0-System} by $u_{\eps}$. Then integration by parts and
 Young's inequality yield
 \begin{align*}
  &\frac{1}{2}\frac{d}{dt}\io u_{\eps}^2=-\io(u_{\eps}+1)^{m_1-1}\vert\nabla u_{\eps}\vert^2
  -\io\frac{u_{\eps}\nabla u_{\eps}\cdot\nabla v_{\eps}}{1+\eps\vert\nabla v_{\eps}\vert}\\
  \leq&-\io(u_{\eps}+1)^{m_1-1}\vert\nabla u_{\eps}\vert^2
  +\io u_{\eps}\vert\nabla u_{\eps}\vert\cdot\vert\nabla v_{\eps}\vert
 \end{align*}
 Since $(u_\eps+1)^{m_1-1}\geq u_\eps$ and $(u_\eps+1)^{m_1-1}\geq 1$, the first integral can be estimated by
 \begin{align*}
  -\io(u_\eps+1)^{m_1-1}\vert\nabla u_\eps\vert^2\leq-\frac{1}{4}\io(u_\eps+1)^{m_1-1}\vert\nabla u_\eps\vert^2
  -\frac{1}{2}\io u_\eps\vert\nabla u_\eps\vert^2-\frac{1}{4}\io\vert\nabla u_\eps\vert^2\quad
  \text{in }(0,T_{max,\eps}),
  \end{align*}
  for each $\eps>0$. Moreover, Young's inequality yields
\begin{align*}
 \io u_\eps\vert\nabla u_\eps\vert\cdot\vert\nabla v_\eps\vert\leq\io u_\eps\vert\nabla v_\eps\vert^2
 +\frac{1}{4}\io u_\eps\vert\nabla u_\eps\vert^2 \quad\text{in }(0,T_{max,\eps})\text{ for each }\eps>0.
\end{align*}
Hence,
\begin{align*}
 \frac{1}{2}\frac{d}{dt}\io u_\eps^2\leq-\frac{1}{4}\io(u_{\eps}+1)^{m_1-1}\vert\nabla u_{\eps}\vert^2
  -\frac{1}{2}\io u_{\eps}\vert\nabla u_{\eps}\vert^2-\frac{1}{4}\io\vert\nabla u_{\eps}\vert^2
  +\frac{1}{4}\io u_{\eps}\vert\nabla u_{\eps}\vert^2+\io u_{\eps}\vert\nabla v_{\eps}\vert^2,
\end{align*}
in $(0,T_{max,\eps})$ for each $\eps>0$.
By Lemma \ref{L2 Raum Zeit Beschränktheit von nabla v}
and due to \eqref{Voraussetzungen an die approximierenden Anfangsdaten}, there is $C>0$ such that
for all $T\in(0,T_{max,\eps})$ and for each $\eps>0$
\begin{align*}
 &\frac{1}{2}\io u_{\eps}^2(\cdot,T)+\frac{1}{4}\int_0^T\io(u_{\eps}+1)^{m_1-1}\vert\nabla u_{\eps}\vert^2
 +\frac{1}{4}\int_0^T\io u_{\eps}\vert\nabla u_{\eps}\vert^2
 +\frac{1}{4}\int_0^T\io\vert\nabla u_{\eps}\vert^2\\
 \leq&\frac{1}{2}\io u_{0,\eps}^2+\int_0^T\io u_{\eps}\vert\nabla v_{\eps}\vert^2\leq C(T+1).
\end{align*}
\end{proof}
Moreover, we show that $(u_\eps+1)$ is bounded in $L^{m_1+1}(\Omega\times(0,T))$.

\begin{lem}\label{(u+1)m_1+1 beschränkt in Raum-Zeit Integral}
 Suppose \eqref{Generalvoraussetzungen}, \eqref{Voraussetzungen an die approximierenden Anfangsdaten}
 and let $T\in(0,T_{max,\eps}]\cap(0,\infty)$.
 Then there is $C(T)>0$ such that for all $\eps>0$ the corresponding classical
 solutions $(u_\eps,v_\eps)$ of \eqref{eps>0-System} satisfy
 \begin{align*}
  \int_0^T\io (u_{\eps}+1)^{m_1+1}\leq C(T).
 \end{align*}
\end{lem}

\begin{proof}
 We apply Gagliardo-Nirenberg's inequality to obtain $C_1,C_2>0$ such that for all $\eps>0$
 \begin{align*}
  &\int_0^T\io (u_{\eps}+1)^{m_1+1}=\Vert (u_{\eps}+1)^{\frac{m_1+1}{2}}\Vert_{L^2(\Omega\times(0,T))}^2\\
  \leq &C_1\Vert (u_{\eps}+1)^\frac{m_1+1}{2}\Vert_{L^1(\Omega\times(0,T))}^\frac{2}{n+2}
  \Vert\nabla (u_{\eps}+1)^\frac{m_1+1}{2}\Vert_{L^2(\Omega\times(0,T))}^\frac{n}{n+2}
  +C_2\Vert (u_{\eps}+1)^\frac{m_1+1}{2}\Vert_{L^1(\Omega\times(0,T))}.
 \end{align*}
 Due to Lemma \ref{ueps und veps in Lm1 beschränkt} there is $C_3(T)>0$ such that
 $\Vert (u_{\eps}+1)^\frac{m_1+1}{2}\Vert_{L^1(\Omega\times(0,T))}\leq C_3$ for all $\eps>0$.
 Lemma \ref{L2 Raum Zeit Beschränktheit von nabla u} yields existence of $C_4(T)>0$ such that,
 for every $\eps>0$,
 $\Vert\nabla(u_\eps+1)^\frac{m_1+1}{2}\Vert_{L^2(\Omega\times(0,T))}^\frac{n}{n+2}\leq C_4(T)$. Hence, there is
 $C_5(T)>0$ such that
 \begin{align*}
  \int_0^T\io(u_\eps+1)^{m_1+1}\leq C_5(T)\quad\text{for each }\eps>0.
 \end{align*}
\end{proof}

To derive boundedness of $v_\eps$, we will use the following Alikakos-Moser iteration.
 \begin{lem}\label{Hilfsabschätzung für Beschränktheit von ueps}\cite[Lemma 3.4]{tao2025suppression}
 Let $a\geq 0$ and $b\geq 1$, and suppose that $(\overline{M}_k)_{k\in\N_0}\subseteq[0,\infty)$ is such that
 \begin{align*}
  \overline{M}_k\leq\max\left\{ a^{2^k},b^k\overline{M}^2_{k-1} \right\}\quad\forall k\geq 1.
 \end{align*}
Then
\begin{align*}
 \overline{M}_k^{\frac{1}{2^k}}\leq b^2\cdot\max\{ a,\overline{M}_0 \}\quad\forall k\geq 0.
\end{align*}
\end{lem}

We can now show that the family $(v_\eps)_{\eps>0}$ is bounded in $L^\infty(\Omega\times(0,T))$.
The computation is nearly the same as in the linear case from Tao and Winkler \cite{tao2025suppression}.
Since we want to derive convergence, it is necessary to show that the bound is
independend of $\eps$.

\begin{lem}\label{globale Existenz veps}
 Suppose \eqref{Generalvoraussetzungen} and \eqref{Voraussetzungen an die approximierenden Anfangsdaten}.
 Then there is $C>0$ such that for all $\eps>0$ the corresponding classical
 solutions $(u_\eps,v_\eps)$ of \eqref{eps>0-System} satisfy
 \begin{align*}
  \Vert v_{\eps}(\cdot,t)\Vert_{L^\infty(\Omega)}\leq C\quad\forall t\in(0,T_{max,\eps}).
 \end{align*}
\end{lem}
\begin{proof}
For $l\in\N$ defining
\begin{align*}
 M_{l,\eps}(T):=\max\left\{ 1,\sup_{t\in(0,T)}\io v_{\eps}^{p_l}(\cdot,t) \right\},\quad T\in(0,T_{max,\eps}),
\end{align*}
an application of Lemma \ref{L1 Norm von ueps konstant und L1 Norm von veps beschränkt}
and the choice of $p_l:=2^l$
yield $C_1>0$ by \eqref{Voraussetzungen an die approximierenden Anfangsdaten} such that for every $\eps>0$
\begin{align*}
 M_{0,\eps}(T)=\max\left\{ 1,\sup_{t\in(0,T)}\io v_\eps(\cdot,t) \right\}\leq\overline{M}_{0,\eps}
 :=1+\io u_{0,\eps}+\io v_{0,\eps}\leq 1+C_1\vert\Omega\vert\quad\text{for all } T\in(0,T_{max,\eps}).
\end{align*}
Testing the second equation of \eqref{eps>0-System} by $v_{\eps}^{p_l-1}$ for $l\geq1$
and using Young's inequality imply
\begin{align*}
 \frac{d}{dt}\io v_{\eps}^{p_l}+\io v_{\eps}^{p_l}
 +2\io\vert\nabla v_{\eps}^\frac{p_l}{2}\vert^2\leq p_l\io v_{\eps}^{p_l}u_{\eps}+\io u_{\eps}
 \quad\text{for all } t\in(0,T_{max,\eps}),\:\eps>0.
\end{align*}
Note that for $m_1>\frac{n}{2}$, it holds $\frac{2n(m_1+1)}{(2+n)m_1}<2$.
Then Hölder's, Gagliardo-Nirenberg's, Young's inequality in combination with
Lemma \ref{ueps und veps in Lm1 beschränkt} show the existence of $C_2,C_3>0$ such that
\begin{align*}
 &p_l\io v_{\eps}^{p_l}u_{\eps}\leq p_l\Vert u_{\eps}\Vert_{L^{m_1}(\Omega)}
 \Vert v_{\eps}^\frac{p_l}{2}\Vert_{L^\frac{2m_1}{m_1-1}(\Omega)}^2\\
 \leq&C_2 p_l\left( C_3\Vert\nabla v_{\eps}^\frac{p_l}{2}\Vert_{L^2(\Omega)}^\frac{n(m_1+1)}{(2+n)m_1}
 \Vert v_{\eps}^\frac{p}{2}\Vert_{L^1(\Omega)}^\frac{2m_1-n}{(2+n)m_1}
 +C_3\Vert v_{\eps}^\frac{p_l}{2}\Vert_{L^1(\Omega)} \right)^2\\
 \leq&2C_2C_3^2p_l\Vert\nabla v_{\eps}^\frac{p_l}{2}\Vert_{L^2(\Omega)}^\frac{2n(m_1+1)}{(2+n)m_1}
 \Vert v_{\eps}^\frac{p_l}{2}\Vert_{L^1(\Omega)}^\frac{4m_1-2n}{(2+n)m_1}
 +2C_2C_3^2p_l\Vert v_{\eps}^\frac{p_l}{2}\Vert_{L^1(\Omega)}^2\\
 \leq&2\Vert\nabla v_{\eps}^\frac{p_l}{2}\Vert_{L^2(\Omega)}^2
 +C_4p_l^\kappa\left( \frac{2m_1(2+n)}{n(m_1+1)}\right)^{-\frac{n(1+m_1)}{2m_1-n}}\cdot\frac{2m_1-n}{(2+n)m_1}
 \Vert v_{\eps}^\frac{p_l}{2}\Vert_{L^1(\Omega)}^2\quad\text{for all }\eps>0.
\end{align*}
with $\kappa:=\frac{(2+n)m_1}{2m_1-n}$ and $C_4:=\left( \frac{C_2C_3^2}{4} \right)^\kappa$. Hence,
\begin{align*}
p_l\io v_{\eps}^{p_l}u_{\eps}\leq 2\io\vert\nabla v_{\eps}^\frac{p_l}{2}\vert^2+C_5p_l^\kappa M_{l-1,\eps}^2(T)
\quad\text{for all } T\in(0,T_{max,\eps}),\:\eps>0,
\end{align*}
with $C_5:=C_4\left( \frac{2m_1(2+n)}{n(m_1+1)}\right)^{-\frac{n(1+m_1)}{2m_1-n}}\cdot\frac{2m_1-n}{(2+n)m_1}$.
Since $M_{l-1,\eps}^2(T)p_l^\kappa\geq 1$,
we can apply Lemma \ref{L1 Norm von ueps konstant und L1 Norm von veps beschränkt} to
find a constant $C_6>0$ such that
\begin{align*}
 \io u_{\eps}\leq C_6\leq C_6 p_l^\kappa M_{l-1,\eps}^2(T)\quad\text{for all } t\in(0,T),\:\eps>0.
\end{align*}
Hence,
\begin{align*}
 \frac{d}{dt}\io v_{\eps}^{p_l}+\io v_{\eps}^{p_l}\leq(C_5+C_6)p_l^\kappa M_{l-1,\eps}^2(T)
 \quad\text{for all } t\in(0,T),\:\eps>0.
\end{align*}
This implies
\begin{align*}
 \io v_{\eps}^{p_l}(\cdot,t)\leq\max\left\{ \io v_{0,\eps}^{p_l},(C_5+C_6)p_l^\kappa M_{l-1,\eps}^2(T) \right\}
 \quad\text{for all } t\in(0,T),\:\eps>0.
\end{align*}
Then we have
\begin{align*}
 M_{l,\eps}(T)\leq\max\left\{ 1,\io v_{0,\eps}^{p_l},(C_5+C_6)p_l^\kappa M_{l-1,\eps}^2(T) \right\}
 \leq\max\{ C_1^{p_l}\vert\Omega\vert,(C_5+C_6+1)p_l^\kappa M_{l-1,\eps}^2(T) \},
\end{align*}
 for all $T\in(0,T_{max,\eps})$ and for all $\eps>0$. Now let
\begin{align*}
 \overline{M}_{l,\eps}:=\sup_{T\in(0,T_{max,\eps})} M_{l,\eps}(T),\quad l\in\N,
\end{align*}
then $\overline{M}_{l,\eps}<\infty$ for all $l\geq 0$ and for all $\eps>0$, with
\begin{align*}
 \overline{M}_{l,\eps}
 \leq\max\left\{ (C_1(1+\vert\Omega\vert))^{2^l},((C_5+C_6+1)\cdot 2^\kappa)^l\overline{M}_{l-1,\eps}^2 \right\}
 \quad\text{for all } l\geq 1,\text{ for all }\eps>0.
\end{align*}
Lemma \ref{Hilfsabschätzung für Beschränktheit von ueps} implies
\begin{align*}
 \overline{M}_{l,\eps}^\frac{1}{2^l}\leq((C_5+C_6+1)\cdot 2^\kappa)^2
 \cdot\max\{ \overline{M}_{0,\eps},C_1(1+\vert\Omega\vert) \},
\end{align*}
for all $l\geq1$ and for all $\eps>0$, from which then boundedness follows by taking the limit $l\to\infty$.
\end{proof}

\section{Global existence and boundedness for fixed $\eps>0$}
In this chapter we show boundedness of solutions $(u_\eps,v_\eps)$ of \eqref{eps>0-System} for fixed $\eps>0$,
from which we can conclude global existence,
by means of the extensibility criterion \eqref{Fortsetzungskriterium in Linfty}.
Therefore,
it remains to show that $u_\eps$ is bounded in $L^\infty(\Omega\times(0,\infty))$.
However, note that this bound on the $L^\infty$-norm of $u_\eps$ still depends on $\eps$.

\begin{lem}\label{globale Existenz ueps}
 Suppose \eqref{Generalvoraussetzungen} and \eqref{Voraussetzungen an die approximierenden Anfangsdaten}.
 Let $\eps>0$ and $(u_\eps,v_\eps)$ be a classical solution of \eqref{eps>0-System}.
 Then there is $C(\eps)>0$ such that
 \begin{align*}
  \Vert u_\eps(\cdot,t)\Vert_{L^\infty(\Omega)}\leq C(\eps)
 \end{align*}
for all $t\in(0,T_{\max,\eps})$.
\end{lem}

\begin{proof}
 We test the first equation of \eqref{eps>0-System} by $u_\eps^{p-1}$, integrate by parts and apply Hölder's and Young's inequality.
 \begin{align*}
  &\frac{1}{p}\frac{d}{dt}\io u_\eps^p=-(p-1)\io u_\eps^{p-2}(u_\eps+1)^{m_1-1}\vert\nabla u_\eps\vert^2
  +(p-1)\io u_\eps^{p-1}\frac{\nabla u_\eps\cdot\nabla v_\eps}{1+\eps\vert\nabla v_\eps\vert}\\
  \leq&-\frac{p-1}{2}\io u_\eps^{p-2}(u_\eps+1)^{m_1-1}\vert\nabla u_\eps\vert^2
  +\frac{p-1}{2}\io u_\eps^{p-m_1+1}
  \frac{\vert\nabla v_\eps\vert^2}{(1+\eps\vert\nabla v_\eps\vert)^2}\\
  \leq&-\frac{p-1}{2}\io u_\eps^{p-2}(u_\eps+1)^{m_1-1}\vert\nabla u_\eps\vert^2
  +\frac{p-1}{2\eps^2}\io u_\eps^{p-m_1+1}\\
  \leq& -\frac{2(p-1)}{p^2}\io\vert\nabla u_\eps^{\frac{p}{2}}\vert^2
  +\frac{p-1}{2\eps^2}\left( \io u_\eps^p \right)^\frac{p-m_1+1}{p}
  \left( \io 1 \right)^\frac{m_1-1}{p}\\
  \leq&-\frac{1}{p}\io\vert\nabla u_\eps^\frac{p}{2}\vert^2+\frac{(p-1)(p-m_1+1)}{2\eps^2}\io u_\eps^p
  +\frac{(p-1)(m_1-1)}{2\eps^2}
  \vert\Omega\vert\\
  \leq&-\frac{1}{p}\io\vert\nabla u_\eps^\frac{p}{2}\vert^2+\frac{p^2}{2\eps^2}\io u_\eps^p
  +\frac{pm_1}{2\eps^2}\vert\Omega\vert,
 \end{align*}
since $-\frac{2(p-1)}{p^2}\leq-\frac{1}{p}$ for $p\geq 2$. Hence,
\begin{align*}
 \frac{1}{p}\frac{d}{dt}\io u_\eps^p+\frac{1}{p}\io u_\eps^p
 \leq-\frac{1}{p}\io\vert\nabla u_\eps^\frac{p}{2}\vert^2
 +\frac{p^2(1+2\eps^2)}{2\eps^2}\io u_\eps^p+\frac{pm_1}{2\eps^2}\vert\Omega\vert.
\end{align*}
Due to Gagliardo-Nirenberg's
and Young's inequalitywe find a constant $C_1>0$
such that
\begin{align*}
 &\io u_\eps^p=\Vert u_\eps^\frac{p}{2}\Vert_{L^2(\Omega)}^2
 \leq \left(C_1\Vert\nabla u_\eps^\frac{p}{2}\Vert_{L^2(\Omega)}^\frac{n}{n+2}
 \Vert u_\eps^\frac{p}{2}\Vert_{L^1(\Omega)}^\frac{2}{n+2}
 +C_1\Vert u_\eps^\frac{p}{2}\Vert_{L^1(\Omega)}\right)^2\\
 \leq&2C_1^2\Vert\nabla u_\eps^\frac{p}{2}\Vert_{L^2(\Omega)}^\frac{2n}{n+2}\Vert u_\eps^\frac{p}{2}\Vert_{L^1(\Omega)}^\frac{4}{n+2}
 +2C_1^2\Vert u_\eps^\frac{p}{2}\Vert_{L^1(\Omega)}^2\\
 \leq &\frac{2\eps^2}{p^3(1+2\eps^2)}\Vert\nabla u_\eps^\frac{p}{2}\Vert_{L^2(\Omega)}^2
 +\left(\frac{2}{n+2}\left( \frac{2(n+2)\eps^2}{p^3n(1+2\eps^2)} \right)^{-\frac{n}{2}}+2C_1^2\right)
 \Vert u_\eps^\frac{p}{2}\Vert_{L^1(\Omega)}^2\\
 \leq&\frac{2\eps^2}{p^3(1+2\eps^2)}\Vert\nabla u_\eps^\frac{p}{2}\Vert_{L^2(\Omega)}^2
 +p^{3n}C_\eps\Vert u_\eps^\frac{p}{2}\Vert_{L^1(\Omega)}^2 \quad\text{on }(0,T_{max,\eps}),
\end{align*}
where $C_\eps:=\frac{2(1+2\eps^2)^\frac{n}{2}}{(n+2)(2^\frac{n}{2}\eps^n)}+2C_1^2$.
For $l\in\N$, let us define $p_l:= 2^l$. Then we have
\begin{align*}
 \frac{d}{dt}\io u_\eps^{p_l}+\io u_\eps^{p_l}\leq p_l^{4n}C_\eps\left( \io u_\eps^\frac{p_l}{2} \right)^2
 +\frac{p_l^{4n}m_1}{2\eps^2}\vert\Omega\vert
\end{align*}
We define for $\eps>0$
\begin{align*}
 M_{l,\eps}(T):=\max\left\{ 1,\sup_{t\in(0,T)}\io u_\eps^{p_l}(\cdot,t) \right\},\quad T\in(0,T_{\max,\eps}).
\end{align*}
Due to $u_{0,\eps}\in C^{2+\alpha}(\overline{\Omega})$ (cf. \eqref{Voraussetzungen an die approximierenden Anfangsdaten})
and $C^{2+\alpha}(\overline{\Omega})\hookrightarrow L^1(\Omega)$, there is $C_2>0$ such that
\begin{align*}
 M_{0,\eps}(T)\leq\overline{M}_{0,\eps}:=1+\io u_{0,\eps}\leq 1+C_2\vert\Omega\vert
 \quad\text{for all }T\in(0,T_{max,\eps}),\text{ for all }\eps>0.
\end{align*}
Since $M_{l-1,\eps}(T)\geq 1$ for all $l\geq 1$,
\begin{align*}
 \frac{d}{dt}\io u_\eps^{p_l}+\io u_\eps^{p_l}\leq p_l^{4n}\tilde{C}_\eps M_{l-1}^2(T)
 \quad\text{for all }T\in(0,T_{max,\eps}),
\end{align*}
where $\tilde{C}_\eps:=C_\eps+\frac{m_1}{2\eps^2}$. This implies
\begin{align*}
 \io u_\eps^{p_l}(\cdot,t)\leq\max\left\{ \io u_0^{p_l}, p_l^{4n}\tilde{C}_\eps M_{l-1}^2(T) \right\},
\end{align*}
for all $t\in(0,T)$. Hence,
\begin{align*}
 M_{l,\eps}(T)\leq\max\left\{ 1,\io u_0^{p_l},p_l^{4n}\tilde{C}_\eps M_{l-1,\eps}(T) \right\}
 \leq\max\left\{ C_2^{p_l}\vert\Omega\vert,p_l^{4n}(\tilde{C}_\eps+1)M_{l-1,\eps}^2(T) \right\},
\end{align*}
for all $T\in(0,T_{max,\eps})$, because $1\leq p_l^{4n}M_{l-1,\eps}^2$.
With $\overline{M}_{l,\eps}:=\sup_{T\in(0,T_{\max,\eps})}M_{l,\eps}(T)$, we conclude
\begin{align*}
 \overline{M}_{l,\eps}
 \leq\max\left\{ \left(C_2(1+\vert\Omega\vert)\right)^{p_l}, p_l^{4n}(\tilde{C}_\eps+1)\overline{M}_{l-1,\eps}^2 \right\},
\end{align*}
hence by Lemma \ref{Hilfsabschätzung für Beschränktheit von ueps},
\begin{align*}
 \overline{M}_{l,\eps}^{\frac{1}{2^l}}\leq 2^{8n}(\tilde{C}_\eps+1)
 \cdot\max\{ \overline{M}_{0,\eps},C_2(1+\vert\Omega\vert) \}
 \quad\text{for all }l\geq 1,
\end{align*}
from which boundedness of $u_\eps$ follows by taking the limit $l\to\infty$.
\end{proof}
Due to that Alikakos-Moser iteration, we can conclude $T_{max,\eps}=\infty$ for all $\eps>0$.
The next step will be deriving bounds for the time derivative to apply the Aubin-Lions-compactness criterion.
This will be an important
ingredience to obtain norm convergence in $L^2(\Omega\times(0,T))$ for $u_\eps$ and $v_\eps$.

\begin{lem}\label{Zeitableitung von ueps im Dual eines Sobolevraums beschränkt}
Suppose \eqref{Generalvoraussetzungen} and \eqref{Voraussetzungen an die approximierenden Anfangsdaten}.
Let $T>0$ and $q>\max\{n,6\}$. Then there is $C(T)>0$ such that for all $\eps>0$ the corresponding
classical solutions $(u_\eps,v_\eps)$ satisfy
 \begin{align*}
  \Vert u_{\eps t}\Vert_{L^1((0,T);\: (W_0^{3,q}(\Omega))^*)}\leq C(T).
 \end{align*}
\end{lem}

\begin{proof}
 Let $\varphi\in C^2(\overline{\Omega})$. Multiplying the first equation of \eqref{eps>0-System} by $\varphi$
 and integrating over $\Omega$ yield the existence of $C_1>0$ due to Lemma \ref{ueps und veps in Lm1 beschränkt}
 and due to the embedding $W^{3,q}(\Omega)\hookrightarrow W^{2,\infty}(\Omega)$, since $q>n$,
 and $C_2>0$ due to Lemma \ref{ueps und veps in Lm1 beschränkt} such that
 \begin{align*}
&\left\vert\io u_{\eps t}\varphi\right\vert
\leq\frac{1}{m_1}\left\vert\io(u_{\eps}+1)^{m_1}\Delta\varphi\right\vert
 +\left\vert\io u_{\eps}\frac{\nabla v_{\eps}}{1+\eps\vert\nabla v_{\eps}\vert}\cdot\nabla\varphi\right\vert\\
 \leq&\frac{1}{m_1}\Vert\Delta\varphi\Vert_{L^\infty(\Omega\times(0,T))}\io(u_{\eps}+1)^{m_1}
 +\Vert \nabla\varphi\Vert_{L^{\frac{2m_1}{m_1-2}}(\Omega)}
 \left( \io u_{\eps}^{m_1} \right)^\frac{1}{m_1}\left( \io\vert\nabla v_{\eps}\vert^2 \right)^\frac{1}{2}\\
 \leq&C_1\Vert\varphi\Vert_{W^{3,q}(\Omega)}
 +C_2\left( \io\vert\nabla v_{\eps}\vert^2 \right)^\frac{1}{2}\Vert\varphi\Vert_{W^{3,q}(\Omega)}.
 \end{align*}
 This implies
\begin{align*}
 \Vert u_{\eps t}(t)\Vert_{(W_0^{3,q}(\Omega))^*}\leq &C_1
 +C_2\left( \io\vert\nabla v_{\eps}\vert^2 \right)^\frac{1}{2}\quad\text{for all }t\in(0,T).
\end{align*}
Hence, due to Lemma \ref{L2 Raum Zeit Beschränktheit von nabla v} there is $C_3(T)>0$ such that for all $\eps>0$
\begin{align*}
 \int_0^T\Vert u_{\eps t}(t)\Vert_{(W_0^{3,q}(\Omega))^*}\:dt\leq C_1T
 +C_2T^{\frac{1}{2}}\left( \int_0^T\io\vert\nabla v_{\eps}\vert^2 \right)^\frac{1}{2}\leq C_3(T).
\end{align*}
\end{proof}
A similar result can be obtained for the second solution component.

\begin{lem}\label{Zeitableitung von veps im Dual eines Sobolevraums beschränkt}
 Suppose \eqref{Generalvoraussetzungen} and \eqref{Voraussetzungen an die approximierenden Anfangsdaten}.
 Let $T>0$ and $q>\max\{n,6\}$. Then there is $C(T)>0$ such that for all
 $\eps>0$ the corresponding classical solutions $(u_\eps,v_\eps)$ satisfy
 \begin{align*}
  \Vert v_{\eps t}\Vert_{L^1((0,T);\:(W_0^{2,q}(\Omega))^*)}\leq C(T).
 \end{align*}
\end{lem}

\begin{proof}
 Let $\psi\in C^1(\overline{\Omega})$. Multiplying the second equation of \eqref{eps>0-System} by $\psi$
 and integrating over $\Omega$ yield existence of $C_1>0$ due to Lemma \ref{globale Existenz veps} and
 $C_2,C_3>0$ due to Lemma \ref{ueps und veps in Lm1 beschränkt} such that for all $\eps>0$
\begin{align*}
  &\left\vert\io v_{\eps t}\psi\right\vert
  \leq\left\vert \io(v_{\eps}+1)^{m_2-1}\nabla v_{\eps}\cdot\nabla\psi \right\vert
  +\left\vert \io u_{\eps}\nabla v_{\eps}\cdot\nabla\psi \right\vert+\left\vert \io v_{\eps}\psi \right\vert
  +\left\vert \io u_{\eps}\psi \right\vert\\
  \leq&\Vert(v_{\eps}+1)\Vert_{L^\infty(\Omega)}^{m_2-1}
  \left( \io\vert\nabla v_{\eps}\vert^\frac{6}{5} \right)^\frac{5}{6}
  \left( \io\vert\nabla\psi\vert^6 \right)^\frac{1}{6}
  +\Vert\nabla\psi\Vert_{L^{\frac{2m_1}{m_1-2}}(\Omega)}
  \left( \io u_{\eps}^{m_1} \right)^\frac{1}{m_1}\left( \io\vert\nabla v_{\eps}\vert^2 \right)^\frac{1}{2}\\
  &+\left( \io v_{\eps}^\frac{6}{5} \right)^\frac{5}{6}\left( \io\psi^6 \right)^\frac{1}{6}
  +\left( \io u_{\eps}^\frac{6}{5} \right)^\frac{5}{6}\left( \io\psi^6 \right)^{\frac{1}{6}}\\
  \leq&C_1\left( \io\vert\nabla v_{\eps}\vert^\frac{6}{5} \right)^\frac{5}{6}\Vert\psi\Vert_{W^{2,q}(\Omega)}
  +C_2\left( \io\vert\nabla v_{\eps}\vert^2 \right)^\frac{1}{2}\Vert\psi\Vert_{W^{2,q}(\Omega)}
  +C_3\Vert\psi\Vert_{W^{2,q}(\Omega)}.
 \end{align*}
Hence, by Lemma \ref{L2 Raum Zeit Beschränktheit von nabla v} there is $C_4(T)>0$ such that for all $t\in(0,T)$
and for all $\eps>0$
\begin{align*}
 \Vert v_{\eps t}(t)\Vert_{(W_0^{2,q}(\Omega))^*}
 \leq C_1\left( \io\vert\nabla v_{\eps}\vert^\frac{6}{5} \right)^\frac{5}{6}
 +C_2\left( \io\vert\nabla v_{\eps}\vert^2 \right)^\frac{1}{2}+C_3.
\end{align*}
This implies
\begin{align*}
 \int_0^T\Vert v_{\eps t}(t)\Vert_{(W_0^{2,q}(\Omega))^*}\:dt
 \leq C_1T^\frac{1}{6}\left( \int_0^T\io\vert\nabla v_{\eps}\vert^\frac{6}{5} \right)^
 \frac{5}{6}+C_2T^\frac{1}{2}\left( \int_0^T\io\vert\nabla v_{\eps}\vert^2 \right)^\frac{1}{2}+C_3T
 \leq C_4(T),
 \end{align*}
 for all $t\in(0,T)$ and for all $\eps>0$.
\end{proof}
 With these bounds we are now able to show some convergence results.

 \section{Convergence}
Due to the bounds obtained in the previous chapters, we can now show some convergence results for
solutions of \eqref{eps>0-System}.
\begin{lem}\label{schwache Konvergenzen von u v und Gradienten}
Suppose \eqref{Generalvoraussetzungen} and \eqref{Voraussetzungen an die approximierenden Anfangsdaten}.
Let $(\eps_k)_{k\in\N}\subseteq(0,\infty)$ be a nullsequence and let
$((u_{\eps_k},v_{\eps_k}))_{k\in\N}$ be a sequence of the corresponding classical solutions of \eqref{eps>0-System}.
Moreover, let $T\in(0,\infty)$.
Then there are a subsequence $((u_{\eps_{k_j}},v_{\eps_{k_j}}))_{j\in\N}$ and a pair of functions
 \begin{align*}
   (u,v)\in \left(L^\infty((0,\infty);L^{m_1}(\Omega))\cap L^2((0,T);\:W^{1,2}(\Omega))\right)\times\left(
 L^\infty(\Omega\times(0,T))\cap L^2((0,T);\:W^{1,2}(\Omega))\right)
 \end{align*}
 such that
 \begin{align}
  &u_{\eps_{k_j}}\overset{*}{\rightharpoonup} u\quad\text{in }L^\infty((0,T);\:L^{m_1}(\Omega)),
  \label{schwache Konvergenz in Lm1}\\
  &v_{\eps_{k_j}}\overset{*}{\rightharpoonup} v\quad\text{in }L^\infty((0,T);\:L^p(\Omega))
  \quad\forall p\in(1,\infty],\label{schwache Konvergenz von v in Lp}\\
  &\nabla u_{\eps_{k_j}}\rightharpoonup\nabla u\quad\text{in }L^2(\Omega\times(0,T)),\label{nabla ueps schwach nabla u}\\
  &\nabla v_{\eps_{k_j}}\rightharpoonup\nabla v\quad\text{in }L^2(\Omega\times(0,T)),\label{nabla veps schwach nabla v}
 \end{align}
as $j\to\infty$.
\end{lem}
\begin{proof}
 The convergence statements \eqref{schwache Konvergenz in Lm1} and \eqref{schwache Konvergenz von v in Lp}
 are direct consequences of the Banach-Alaoglu theorem, along with the boundedness obtained in
 Lemma \ref{ueps und veps in Lm1 beschränkt} and Lemma \ref{globale Existenz veps}.
 For the gradients, due to Lemma \ref{L2 Raum Zeit Beschränktheit von nabla u},
 there is $z\in L^2(\Omega\times(0,T))$ such that, possibly along a further subsequence,
 $\nabla u_{\eps_{k_j}}\rightharpoonup z$ in $L^2(\Omega\times(0,T))$.
 By \eqref{schwache Konvergenz in Lm1}, this implies
 \begin{align*}
  -\int_0^T\io u\nabla\varphi\leftarrow-\int_0^T\io u_{\eps_{k_j}}\nabla\varphi
  =\int_0^T\io\varphi\nabla u_{\eps_{k_j}}\to\int_0^T\io z\varphi
 \end{align*}
for all $\varphi\in C_c^\infty(\Omega\times(0,T))$. Hence, $z=\nabla u$ in $L^2(\Omega\times(0,T))$.
By the same argument with Lemma \ref{L2 Raum Zeit Beschränktheit von nabla v}, it follows that
$\nabla v_{\eps_{k_j}}\rightharpoonup\nabla v$ in $L^2(\Omega\times(0,T))$.
\end{proof}

Since we want to have weak convergence of terms like $u_{\eps_k}\nabla v_{\eps_k}$,
we need norm convergence of one factor. To derive
that, we recall the following statement of the Aubin-Lions compactness criterion.
\begin{lem}\label{Relativ kompakte Mengen in Lp Simon}\cite[Chapter 8, Corollary 4]{simon1986}
 Let $X,Y,Z$ be Banach spaces with $X\overset{c}{\hookrightarrow}Y\hookrightarrow Z$ and $p\in[1,\infty)$.
 Moreover, let $\mathcal{F}\subseteq L^p((0,T);\:X)$ be a bounded set and let
 $\frac{\partial\mathcal{F}}{\partial t}=\left\{ \frac{\partial f}{\partial t}\mid f\in\mathcal{F}\right\}$ be bounded in
 $L^1((0,T);\:Z)$. Then $\mathcal{F}$ is relatively compact in $L^p((0,T);\:Y)$.
\end{lem}
Thanks to the previous result we can now show norm convergence of $u_\eps$ and $v_\eps$ in $L^2(\Omega\times(0,T))$.

\begin{lem}\label{Normkonvergenzen von ueps und veps}
Suppose \eqref{Generalvoraussetzungen} and \eqref{Voraussetzungen an die approximierenden Anfangsdaten}. Let
$(\eps_k)_{k\in\N}\subseteq(0,\infty)$ be a nullsequence and $((u_{\eps_{k}},v_{\eps_{k}}))_{k\in\N}$
a corresponding sequence of classical solutions of \eqref{eps>0-System}.
 Then there are a subsequence $(u_{\eps_{k_j}},v_{\eps_{k_j}})_{j\in\N}$ and a pair of functions
 \begin{align*}
  (u,v)\in &\left(L^\infty((0,\infty);L^{m_1}(\Omega))\cap L^2_{loc}([0,\infty);\:W^{1,2}(\Omega))\right)\times\\
  &\left(L^\infty(\Omega\times(0,\infty))\cap L^2_{loc}([0,\infty);\:W^{1,2}(\Omega))\right)
 \end{align*}
 such that for every $T>0$
 \begin{align}
  &v_{\eps_{k_j}}\to v\quad\text{in }L^2(\Omega\times(0,T)),\label{veps to v in L2}\\
  &u_{\eps_{k_j}}\to u\quad\text{in }L^2(\Omega\times(0,T)),\label{ueps to u in L2}\\
  &(u_{\eps_{k_j}}+1)^{m_1}\rightharpoonup(u+1)^{m_1}\quad\text{in }L^{\frac{m_1+1}{m_1}}(\Omega\times(0,T)),
  \label{(u+1)m1 schwach in Lm1+1m1}
 \end{align}
\end{lem}
\begin{proof}
 Due to the embedding $W^{1,2}(\Omega)\overset{c}{\hookrightarrow}L^2(\Omega)\hookrightarrow (W_0^{2,q}(\Omega))^*$ with $q>n$,
 Lemma \ref{L2 Raum Zeit Beschränktheit von nabla v} and Lemma \ref{Zeitableitung von veps im Dual eines Sobolevraums beschränkt},
  we can apply Lemma \ref{Relativ kompakte Mengen in Lp Simon}.
Hence, there are a subsequence $(v_{\eps_{k_j}})_{j\in\N}$ such that
\begin{align*}
 v_{\eps_{k_j}}\to v\quad\text{in }L^2(\Omega\times(0,T))\quad\text{as }j\to\infty.
\end{align*}
Due to the embedding $W^{1,2}(\Omega)\overset{c}{\hookrightarrow}L^2(\Omega)\hookrightarrow(W_0^{3,q}(\Omega))^*$ with $q>n$,
Lemma \ref{L2 Raum Zeit Beschränktheit von nabla u} and Lemma
\ref{Zeitableitung von ueps im Dual eines Sobolevraums beschränkt},
we can apply Lemma \ref{Relativ kompakte Mengen in Lp Simon} again. Hence, there is a subsequence
$(u_{\eps_{k_j}})_{j\in\N}$ such that
\begin{align*}
 u_{\eps_{k_j}}\to u\quad\text{in }L^2(\Omega\times(0,T))\quad\text{as }j\to\infty.
\end{align*}
This implies for a subsequence $u_{\eps_{k_{j_l}}}\to u$ pointwise almost everywhere in $\Omega\times(0,T)$
as $l\to\infty$.
Since the function $s\mapsto(s+1)^{m_1}$ is continuous, it holds $(u_{\eps_{k_{j_l}}}+1)^{m_1}\to
(u+1)^{m_1}$ pointwise almost everywhere in $\Omega\times(0,T)$ as $l\to\infty$. Moreover, by Lemma
\ref{(u+1)m_1+1 beschränkt in Raum-Zeit Integral} and the Eberlein-Smulian theorem there
is $z\in L^{\frac{m_1+1}{m_1}}(\Omega\times(0,T))$ such that $(u_{\eps_{k_{j_l}}}+1)^{m_1}
\rightharpoonup z$ in $L^{\frac{m_1+1}{m_1}}(\Omega\times(0,T))$. Due to pointwise convergence, it holds
$z=(u+1)^{m_1}$ almost everywhere in $\Omega\times(0,T)$.
\end{proof}
Since our goal is to show that the pair of limit functions $(u,v)$ solve
\eqref{Poröse Medien mit nichtlinearer Diffusion} in
a weak sense, we also have to show weak convergence of some mixed terms.
We start that with the main challenge, which is
the weak convergence of the regularization term in $L^2(\Omega\times(0,T))$,
which depends on $\nabla v_\eps$ in a nonlinear manner.
\begin{lem}\label{schwache Konvergenz des Bruchs gegen nabla v}
 Suppose \eqref{Generalvoraussetzungen} and \eqref{Voraussetzungen an die approximierenden Anfangsdaten}.
 Let $(\eps_k)_{k\in\N}\subseteq(0,\infty)$ be a nullsequence
 and let $(u_{\eps_k},v_{\eps_k})_{k\in\N}$ be a sequence of classical solutions
 to \eqref{eps>0-System} and $T>0$.
 Then there is a subsequence $((u_{\eps_{k_j}},v_{\eps_{k_j}}))_{j\in\N}$ such that
 \begin{align*}
  \frac{\nabla v_{\eps_{k_j}}}{1+\eps_{k_j}\vert\nabla v_{\eps_{k_j}}\vert}\rightharpoonup\nabla v
  \quad\text{in }L^2(\Omega\times(0,T))\quad\text{as }j\to\infty.
 \end{align*}
\end{lem}
\begin{proof}
 Due to \eqref{nabla veps schwach nabla v},
 we know $\nabla v_{\eps_{k_j}}\rightharpoonup\nabla v$ in $L^2(\Omega\times(0,T))$. Since
 $\vert\frac{\nabla v_{\eps_{k_j}}}{1+\eps_{k_j}\vert\nabla v_{\eps_{k_j}}\vert}\vert\leq
 \vert\nabla v_{\eps_{k_j}}\vert$ for all $j\in\N$ and due to boundedness of $(\nabla v_{\eps_{k_j}})_{j\in\N}$
 in $L^2(\Omega\times(0,T))$ by Lemma \ref{L2 Raum Zeit Beschränktheit von nabla v},
 there is $z\in L^2(\Omega\times(0,T))$ such that
 $\frac{\nabla v_{\eps_{k_j}}}{1+\eps_{k_j}\vert\nabla v_{\eps_{k_j}}\vert}\rightharpoonup z$ in
 $L^2(\Omega\times(0,T))$.
 It remains to show $z=\nabla v$.
 By Lemma \ref{L2 Raum Zeit Beschränktheit von nabla v}
 we find a constant $C(T)>0$ such that
 \begin{align*}
  &\left\Vert 1-\frac{1}{1+\eps_{k_j}\vert\nabla v_{\eps_{k_j}}\vert}\right\Vert_{L^2(\Omega\times(0,T))}^2
  =\int_0^T\io\left(1-\frac{1}{1+\eps_{k_j}\vert\nabla v_{\eps_{k_j}}\vert}\right)^2\\
  =&\int_0^T\io\left( \frac{\eps_{k_j}\vert\nabla v_{\eps_{k_j}}\vert}{1+\eps_{k_j}
  \vert\nabla v_{\eps_{k_j}}\vert} \right)^2
  =\eps_{k_j}^2\int_0^T\io\frac{\vert\nabla v_{\eps_{k_j}}\vert^2}{(1+\eps_{k_j}
  \vert\nabla v_{\eps_{k_j}}\vert)^2}\\
  \leq&\eps_{k_j}^2\int_0^T\io\vert\nabla v_{\eps_{k_j}}\vert^2
  \leq \eps_{k_j}^2C(T)\to 0\quad\text{as }\eps_{k_j}\searrow0.
 \end{align*}
Hence, $\frac{1}{1+\eps_{k_j}\vert\nabla v_{\eps_{k_j}}\vert}\to 1$ in $L^2(\Omega\times(0,T))$. This implies
\begin{align*}
 \nabla v_{\eps_{k_j}}\cdot\frac{1}{1+\eps_{k_j}\vert\nabla v_{\eps_{k_j}}\vert}\rightharpoonup\nabla v\cdot 1
 \quad\text{in }L^1(\Omega\times(0,T)).
\end{align*}
Due to uniqueness of weak limits, it follows $z=\nabla v$ and hence
$\frac{\nabla v_{\eps_{k_j}}}{1+\eps_{k_j}\vert\nabla v_{\eps_{k_j}}\vert}\rightharpoonup z=\nabla v$
in $L^2(\Omega\times(0,T))$.
\end{proof}

\begin{lem}\label{schwache Konvergenz (ueps+1)m1-1 nabla u}
 Suppose \eqref{Generalvoraussetzungen} and \eqref{Voraussetzungen an die approximierenden Anfangsdaten}.
 Let $(\eps_k)_{k\in\N}\subseteq(0,\infty)$ are nullsequence
 and let $(u_{\eps_k},v_{\eps_k})_{k\in\N}$ be a sequence of classical solutions of \eqref{eps>0-System} and $T>0$.
 Then there is a subsequence $(u_{\eps_{k_j}},v_{\eps_{k_j}})_{j\in\N}$ such that
 \begin{align*}
  (u_{\eps_{k_j}}+1)^{m_1-1}\nabla u_{\eps_{k_j}}\rightharpoonup(u+1)^{m_1-1}\nabla u
  \quad\text{in }L^{\frac{m_1+1}{m_1}}(\Omega\times(0,T)).
 \end{align*}
\end{lem}
\begin{proof}
 Due to Hölder's inequality, Lemma \ref{L2 Raum Zeit Beschränktheit von nabla u}
 and Lemma \ref{(u+1)m_1+1 beschränkt in Raum-Zeit Integral} there is $C(T)>0$ such that
 \begin{align*}
  &\int_0^T\io ((u_{\eps_{k_j}}+1)^{m_1-1}\nabla u_{\eps_{k_j}})^{\frac{m_1+1}{m_1}}
  =\int_0^T\io(u_{\eps_{k_j}}+1)^{\frac{(m_1+1)(m_1-1)}{m_1}}\vert\nabla u_{\eps_{k_j}}\vert^{\frac{m_1+1}{m_1}}\\
  \leq&\int_0^T\left( \io\vert\nabla u_{\eps_{k_j}}\vert^2 \right)^\frac{m_1+1}{2m_1}
  \left( \io(u_{\eps_{k_j}}+1)^{m_1+1} \right)^\frac{m_1-1}{2m_1}\\
  \leq&\left( \int_0^T\io\vert\nabla u_{\eps_{k_j}}\vert^2 \right)^\frac{m_1+1}{2m_1}
  \left( \int_0^T\io (u_{\eps_{k_j}}+1)^{m_1+1} \right)^\frac{m_1+1}{2m_1}\leq C(T).
 \end{align*}
Since $\frac{m_1+1}{m_1}>1$, we can apply the Eberlein-Smulian theorem
and obtain $z\in L^{\frac{m_1+1}{m_1}}(\Omega\times(0,T))$
such that
\begin{align*}
 (u_{\eps_{k_j}}+1)^{m_1-1}\nabla u_{\eps_{k_j}}\rightharpoonup z\quad\text{in }L^\frac{m_1+1}{m_1}(\Omega\times(0,T)).
\end{align*}
Since $(u_{\eps_{k_j}}+1)^{m_1-1}\nabla u_{\eps_{k_j}}=\frac{1}{m_1}(u_{\eps_{k_j}}+1)^{m_1}$ and
due to the convergence statement \eqref{(u+1)m1 schwach in Lm1+1m1}, we have for all
$\varphi\in C_c^\infty(\Omega\times(0,T))$
\begin{align*}
 \int_0^T\io z\varphi\leftarrow\frac{1}{m_1}\int_0^T\io\nabla(u_{\eps_{k_j}}+1)^{m_1}\varphi
 =-\frac{1}{m_1}\int_0^T\io(u_{\eps_{k_j}}+1)^{m_1}\nabla\cdot\varphi
 \to\frac{1}{m_1}\int_0^T\io(u+1)^{m_1}\nabla\cdot\varphi.
\end{align*}
Thus, $z=\frac{1}{m_1}\nabla(u+1)^{m_1}$.
\end{proof}

\section{Weak solutions. Proof of Theorem \ref{Thm globale Existenz}}
We formulate the definition for weak solutions of \eqref{Poröse Medien mit nichtlinearer Diffusion}
explicitly.
\begin{definition}\label{schwache Lösungsdefinition für Grenzsystem}
Suppose \eqref{Generalvoraussetzungen} and let $T\in(0,\infty]$.
A weak solution of \eqref{Poröse Medien mit nichtlinearer Diffusion} in $\Omega\times(0,T)$
is a pair of functions $u\in L^2((0,T);\: L^2(\Omega))\cap L^2_{loc}([0,\infty);\:W^{1,2}(\Omega))$ and
$v\in L^2((0,T);\: L^2(\Omega))\cap L^2_{loc}([0,\infty);\:W^{1,2}(\Omega))$
such that $u,v$ are nonnegative and satisfy
 \begin{align*}
  -\int_0^T\io u\varphi_t-\io u_0\varphi(\cdot,0)=-\int_0^T\io (u+1)^{m_1-1}\nabla u\cdot\nabla \varphi+\int_0^T\io u\nabla v\cdot\nabla
  \varphi,
 \end{align*}
for all $\varphi\in C_c^\infty(\overline{\Omega}\times[0,T))$ and
\begin{align*}
 -\int_0^T\io v\psi_t-\io v_0\psi(\cdot,0)=&-\int_0^T\io(v+1)^{m_2-1}\nabla v\cdot\nabla\psi-\int_0^T\io u\nabla v\cdot\nabla\psi\\
 &-\int_0^T\io v\psi+\int_0^T\io u\psi,
\end{align*}
for all $\psi\in C_c^\infty(\overline{\Omega}\times[0,T))$. If $T=\infty$, we call the solution global.
\end{definition}

Our aim is to show that the functions we obtain from passing to the limit $\eps>0$ are weak solutions of
\eqref{Poröse Medien mit nichtlinearer Diffusion}
in the sense of Definition \ref{schwache Lösungsdefinition für Grenzsystem}.

\begin{lem}\label{Grenzwert löst PDE}
Suppose \eqref{Generalvoraussetzungen} and \eqref{Voraussetzungen an die approximierenden Anfangsdaten}. Let
$(\eps_k)_{k\in\N}\subseteq(0,\infty)$ be a nullsequence and $((u_{\eps_k},v_{\eps_k}))_{k\in\N}$
the sequence of the corresponding classical solutions.
 Then the pair of limit functions $(u,v)$ from Lemma \ref{Normkonvergenzen von ueps und veps}
 solves \eqref{Poröse Medien mit nichtlinearer Diffusion} weakly.
\end{lem}
\begin{proof}
Let $\varphi\in C_c^\infty(\overline{\Omega}\times[0,\infty))$. Then the global classical solutions
$(u_{\eps_k},v_{\eps_k})$ (cf. Lemma \ref{globale Existenz ueps} and Lemma \ref{globale Existenz veps}) satisfy
\begin{align*}
  -\int_0^\infty\io u_{\eps_k}\varphi_t-\io u_{0,\eps_k}\varphi(\cdot,0)
  =&-\int_0^\infty\io (u_{\eps_k}+1)^{m_1-1}\nabla u_{\eps_k}\cdot\nabla \varphi
  +\int_0^\infty\io\frac{u_{\eps_k}\nabla v_{\eps_k}\cdot\nabla\varphi}{1+\eps_k\vert\nabla v_{\eps_k}\vert},
 \end{align*}
 and
 \begin{align*}
-\int_0^\infty\io v_{\eps_k}\varphi_t-\io v_{0,\eps}\varphi(\cdot,0)
=&-\int_0^\infty\io(v_{\eps_k}+1)^{m_2-1}\nabla v_{\eps_k}\cdot\nabla\varphi\\
 &-\int_0^\infty\io u_{\eps_k}\nabla v_{\eps_k}\cdot\nabla\varphi
 -\int_0^\infty\io v_{\eps_k}\varphi+\int_0^\infty\io u_{\eps_k}\varphi,\quad\text{for all }k\in\N.
 \end{align*}
Since $\varphi_t\in L^2(\Omega\times(0,\infty))$ and $\text{supp}\:\varphi_t$ is bounded,
the $L_{loc}^2$-convergence \eqref{ueps to u in L2} implies for a subsequence
\begin{align*}
 \int_0^\infty\io u_{\eps_{k_j}}\varphi_t\to\int_0^\infty\io u\varphi_t,
\end{align*}
as $\eps_{k_j}\searrow 0$. Due to \eqref{Voraussetzungen an die approximierenden Anfangsdaten}, we have
\begin{align*}
 -\io u_{0,\eps_{k_j}}\varphi(\cdot,0)\to-\io u_0\varphi(\cdot,0).
\end{align*}
Since $\nabla\varphi\in L^{m_1}(\Omega\times(0,\infty))$ and $\text{supp }\nabla\varphi$ is bounded, and due to
Lemma \ref{schwache Konvergenz (ueps+1)m1-1 nabla u}, we obtain
\begin{align*}
 \int_0^\infty\io(u_{\eps_{k_j}}+1)^{m_1-1}\nabla u_{\eps_{k_j}}\cdot\nabla\varphi\to
 \int_0^\infty\io(u+1)^{m_1-1}\nabla u\cdot\nabla\varphi.
\end{align*}
Since $\nabla\varphi\in L^\infty(\Omega\times(0,\infty))$ and $\text{supp }\nabla\varphi$ is bounded,
and $u_{\eps_{k_j}}\to u$ in $L_{loc}^2(\Omega\times[0,\infty))$, Lemma
\ref{schwache Konvergenz des Bruchs gegen nabla v} implies
\begin{align*}
 \int_0^\infty\io u_{\eps_{k_j}}\frac{\nabla v_{\eps_{k_j}}}{1+\eps_{k_j}\vert\nabla v_{\eps_{k_j}}\vert}\cdot
 \nabla\varphi\to\int_0^\infty\io u\nabla v\cdot\nabla\varphi.
\end{align*}
For the second equation, we use \eqref{veps to v in L2}
due to $\varphi_t\in L^2(\Omega\times(0,\infty))$ as well as boundedness of $\text{supp }\varphi_t$ and obtain
\begin{align*}
 \int_0^\infty\io v_{\eps_{k_j}}\varphi_t\to\int_0^\infty\io v\varphi_t
\end{align*}
as $\eps_{k_j}\searrow 0$. Since $u_{\eps_{k_j}}\to u$ in $L_{loc}^2(\Omega\times[0,\infty))$
and $\nabla v_{\eps_{k_j}}\rightharpoonup \nabla v$ in $L_{loc}^2(\Omega\times[0,\infty))$,
we have $u_{\eps_{k_j}}\nabla v_\eps\rightharpoonup u\nabla v$ in $L_{loc}^1(\Omega\times[0,\infty))$.
This implies
\begin{align*}
 \int_0^\infty\io u_{\eps_{k_j}}\nabla v_{\eps_{k_j}}\cdot\nabla\varphi
 \to\int_0^\infty\io u\nabla v\cdot\nabla\varphi
\end{align*}
for all $\varphi\in C_c^\infty(\overline{\Omega}\times[0,\infty))$.
Due to $L_{loc}^2$-convergence of $u_{\eps_{k_j}}$ \eqref{ueps to u in L2} and $v_{\eps_{k_j}}$ \eqref{veps to v in L2},
it holds
\begin{align*}
 -\int_0^\infty\io v_{\eps_{k_j}}\varphi+\int_0^\infty\io u_{\eps_{k_j}}\varphi
 \to-\int_0^\infty\io v\varphi+\int_0^\infty\io u\varphi.
\end{align*}
Hence, the pair $(u,v)$ is a weak solution of \eqref{Poröse Medien mit nichtlinearer Diffusion}.
\end{proof}

\begin{proof}[Proof of Theorem \ref{Thm globale Existenz}]
 The global existence follows from Lemma \ref{Normkonvergenzen von ueps und veps}, where
 the limit functions $u,v$ are shown to exist,
 and Lemma \ref{Grenzwert löst PDE} asserting that $(u,v)$ is a weak solution.
 The estimate \eqref{Beschränktheit der Lösungen des Grenzsystems im Thm} is a consequence of the bounds from
 Lemma \ref{ueps und veps in Lm1 beschränkt} and Lemma \ref{globale Existenz veps}
 and of the convergence results \eqref{schwache Konvergenz in Lm1} and \eqref{schwache Konvergenz von v in Lp} of Lemma \ref{schwache Konvergenzen von u v und Gradienten}.
\end{proof}

\section{Acknowledgments}
This research was funded by the Deutsche Forschungsgemeinschaft (DFG, German Research Foundation)
– project no. 552796953 (``Explosions and solvability in cross-diffusive systems of chemotaxis type''),
whose support is gratefully acknowledged.

\phantomsection
\addcontentsline{toc}{section}{References}
 \bibliographystyle{abbrv}
 \bibliography{lit.bib}

\begin{thebibliography}{10}

\bibitem{amann1993nonhomogeneous}
H.~Amann.
\newblock Nonhomogeneous linear and quasilinear elliptic and parabolic boundary
  value problems.
\newblock In {\em Function spaces, differential operators and nonlinear
  analysis}, page 9–126. Springer, 1993.

\bibitem{BELBELTAOWIN2015}
N.~Bellomo, A.~Bellouquid, Y.~Tao, and M.~Winkler.
\newblock Toward a mathematical theory of keller–segel models of pattern
  formation in biological tissues.
\newblock {\em Mathematical Models and Methods in Applied Sciences},
  25:1663--1763, 05 2015.

\bibitem{bellomo2017degenerate}
N.~Bellomo and M.~Winkler.
\newblock A degenerate chemotaxis system with flux limitation: Maximally
  extended solutions and absence of gradient blow-up.
\newblock {\em Communications in Partial Differential Equations},
  42(3):436–473, 2017.

\bibitem{blanchet2009critical}
A.~Blanchet, J.~A. Carrillo, and P.~Laurençot.
\newblock {Critical mass for a Patlak–Keller–Segel model with degenerate
  diffusion in higher dimensions}.
\newblock {\em Calculus of Variations and Partial Differential Equations},
  35(2):133–168, 2009.

\bibitem{Cao_Fuest_2025}
X.~Cao and M.~Fuest.
\newblock {Finite-time blow-up in fully parabolic quasilinear Keller-Segel
  systems with supercritical exponents}, Feb. 2025.
\newblock arXiv:2409.19388v2.

\bibitem{Bo2014}
B.~Chen, Y.~Jia, Y.~Gao, L.~Sanchez, S.~M. Anthony, and Y.~Yu.
\newblock Janus particles as artificial antigen-presenting cells for $t$ cell
  activation.
\newblock {\em ACS Applied Materials \& Interfaces}, 6(21):18435--18439, 10
  2014.

\bibitem{Cieslak_Stinner_2015}
T.~Cieślak and C.~Stinner.
\newblock {New critical exponents in a fully parabolic quasilinear
  Keller–Segel system and applications to volume filling models}.
\newblock {\em Journal of Differential Equations}, 258(6):2080–2113, Mar.
  2015.

\bibitem{MR4026927}
M.~Ding and X.~Zhao.
\newblock {{$L^\sigma$}-measure criteria for boundedness in a quasilinear
  parabolic-parabolic {K}eller-{S}egel system with supercritical sensitivity}.
\newblock {\em Discrete and Continuous Dynamical Systems. Series B. A Journal
  Bridging Mathematics and Sciences}, 24(10):5297–5315, 2019.

\bibitem{Fuest_Lankeit_Tanaka_2023}
M.~Fuest, J.~Lankeit, and Y.~Tanaka.
\newblock {Critical mass phenomena in higher dimensional quasilinear
  {K}eller-{S}egel systems with indirect signal production}.
\newblock {\em Mathematical Methods in the Applied Sciences},
  46(13):14362–14378, 2023.

\bibitem{Hashira_Ishida_Yokota_2018}
T.~Hashira, S.~Ishida, and T.~Yokota.
\newblock {Finite-time blow-up for quasilinear degenerate {K}eller-{S}egel
  systems of parabolic-parabolic type}.
\newblock {\em Journal of Differential Equations}, 264(10):6459–6485, 2018.

\bibitem{MR1627338}
M.~A. Herrero and J.~J.~L. Velázquez.
\newblock A blow-up mechanism for a chemotaxis model.
\newblock {\em Annali della Scuola Normale Superiore di Pisa. Classe di
  Scienze. Serie IV}, 24(4):633–683 (1998), 1997.

\bibitem{MR2013508}
D.~Horstmann.
\newblock From 1970 until present: the {K}eller-{S}egel model in chemotaxis and
  its consequences. {I}.
\newblock {\em Jahresbericht der Deutschen Mathematiker-Vereinigung},
  105(3):103–165, 2003.

\bibitem{Ishida_Seki_Yokota2014}
S.~Ishida, K.~Seki, and T.~Yokota.
\newblock {Boundedness in quasilinear Keller–Segel systems of
  parabolic–parabolic type on non-convex bounded domains}.
\newblock {\em Journal of Differential Equations}, 256(8):2993–3010, Apr.
  2014.

\bibitem{ishida2012global}
S.~Ishida and T.~Yokota.
\newblock {Global existence of weak solutions to quasilinear degenerate
  Keller–Segel systems of parabolic–parabolic type}.
\newblock {\em Journal of Differential Equations}, 252(2):1421–1440, 2012.

\bibitem{KELLER1970399}
E.~F. Keller and L.~A. Segel.
\newblock Initiation of slime mold aggregation viewed as an instability.
\newblock {\em Journal of Theoretical Biology}, 26(3):399–415, 1970.

\bibitem{LSU}
O.~Ladyzenskaja, V.~Solonnikov, and N.~Ural'ceva.
\newblock {\em Linear and Quasi-linear Equations of Parabolic Type}, volume~23
  of {\em Translations of Mathematical Monographs}.
\newblock American Mathematical Society, 1968.

\bibitem{Lankeit_Lankeit_2019}
E.~Lankeit and J.~Lankeit.
\newblock Classical solutions to a logistic chemotaxis model with singular
  sensitivity and signal absorption.
\newblock {\em Nonlinear Analysis: Real World Applications}, 46:421–445, Apr.
  2019.

\bibitem{Lankeit_2020}
J.~Lankeit.
\newblock {Infinite time blow-up of many solutions to a general quasilinear
  parabolic-elliptic {K}eller-{S}egel system}.
\newblock {\em Discrete and Continuous Dynamical Systems. Series S},
  13(2):233–255, 2020.

\bibitem{MR4071449}
J.~Lankeit and M.~Winkler.
\newblock Facing low regularity in chemotaxis systems.
\newblock {\em Jahresbericht der Deutschen Mathematiker-Vereinigung},
  122(1):35–64, 2020.

\bibitem{Laurencot_Mizoguchi_2017}
P.~Laurençot and N.~Mizoguchi.
\newblock Finite time blowup for the parabolic–parabolic {K}eller–{S}egel
  system with critical diffusion.
\newblock {\em Annales de l'Institut Henri Poincaré C, Analyse non linéaire},
  34(1):197–220, Feb. 2017.

\bibitem{PhysRevLett.115.258301}
B.~Liebchen, D.~Marenduzzo, I.~Pagonabarraga, and M.~E. Cates.
\newblock Clustering and pattern formation in chemorepulsive active colloids.
\newblock {\em Physical Review Letters}, 115:258301, Dec 2015.

\bibitem{MR932759}
G.~M. Lieberman.
\newblock Hölder continuity of the gradient of solutions of uniformly
  parabolic equations with conormal boundary conditions.
\newblock {\em Annali di Matematica Pura ed Applicata. Serie Quarta},
  148:77–99, 1987.

\bibitem{Lieberman1996}
G.~M. Lieberman.
\newblock {\em Second order parabolic differential equations}.
\newblock World Scientific Publishing Co., Inc., River Edge, NJ, 1996.

\bibitem{Liu_Wang_2015}
J.-G. Liu and J.~Wang.
\newblock {A Note on $L^\infty$-Bound and Uniqueness to a Degenerate
  Keller-Segel Model}.
\newblock {\em Acta Applicandae Mathematicae}, 142:173–188, 06 2015.

\bibitem{MIMURA20171477}
Y.~Mimura.
\newblock {The variational formulation of the fully parabolic Keller–Segel
  system with degenerate diffusion}.
\newblock {\em Journal of Differential Equations}, 263(2):1477–1521, 2017.

\bibitem{Mizukami2019}
M.~Mizukami, T.~Ono, and T.~Yokota.
\newblock Extensibility criterion ruling out gradient blow-up in a quasilinear
  degenerate chemotaxis system with flux limitation.
\newblock {\em Journal of Differential Equations}, 267(9):5115–5164, Oct.
  2019.

\bibitem{Painter2019}
K.~J. Painter.
\newblock Mathematical models for chemotaxis and their applications in
  self-organisation phenomena.
\newblock {\em Journal of Theoretical Biology}, 481:162–182, Nov. 2019.

\bibitem{simon1986}
J.~Simon.
\newblock Compact sets in the space ${L}^p(0,{T};{B})$.
\newblock {\em Annali di Matematica pura ed applicata}, 146:65--96, 1986.

\bibitem{stinner2024critical}
C.~Stinner and M.~Winkler.
\newblock {A critical exponent in a quasilinear Keller–Segel system with
  arbitrarily fast decaying diffusivities accounting for volume-filling
  effects}.
\newblock {\em Journal of Evolution Equations}, 24(2):26, 2024.

\bibitem{Sugiyama2007}
Y.~Sugiyama.
\newblock {Time global existence and asymptotic behavior of solutions to
  degenerate quasi-linear parabolic systems of chemotaxis}.
\newblock {\em Differential and Integral Equations}, 20:133–180, 01 2007.

\bibitem{Tao_Winkler_2012}
Y.~Tao and M.~Winkler.
\newblock {Boundedness in a quasilinear parabolic–parabolic Keller–Segel
  system with subcritical sensitivity}.
\newblock {\em Journal of Differential Equations}, 252(1):692–715, Jan. 2012.

\bibitem{tao2025suppression}
Y.~Tao and M.~Winkler.
\newblock Suppression of blow-up by local anisotropy of signal production in
  the {K}eller-{S}egel system.
\newblock {\em Journal de Mathématiques Pures et Appliquées}, page 103795,
  2025.

\bibitem{tao2026preventinglpblowuplocal}
Y.~Tao and M.~Winkler.
\newblock Preventing ${L}^p$ blow-up by local anisotropy of signal production
  in the {K}eller-{S}egel system with strongly differing diffusion rates, 2026.
\newblock arXiv:2606.06009.

\bibitem{Tian_Zheng_2016}
M.~Tian and S.~Zheng.
\newblock {Global boundedness versus finite-time blow-up of solutions to a
  quasilinear fully parabolic {K}eller-{S}egel system of two species}.
\newblock {\em Communications on Pure and Applied Analysis}, 15(1):243–260,
  2016.

\bibitem{ScienceDirect:S0021782413000287}
M.~Winkler.
\newblock Finite-time blow-up in the higher-dimensional parabolic–parabolic
  {K}eller–{S}egel system.
\newblock {\em Journal de Mathématiques Pures et Appliquées},
  100(5):748–767, Nov. 2013.

\bibitem{Winkler_2019}
M.~Winkler.
\newblock {Global classical solvability and generic infinite-time blow-up in
  quasilinear Keller–Segel systems with bounded sensitivities}.
\newblock {\em Journal of Differential Equations}, 266(12):8034–8066, June
  2019.

\bibitem{Zhao_2026}
Y.~Zhao.
\newblock Boundedness facilitated by large diffusion rate of cells in the
  {K}eller-{S}egel model with anisotropic signal production.
\newblock {\em Journal of Mathematical Analysis and Applications}, 564(2):Paper
  No. 130915, 11, 2026.

\end{thebibliography}

\end{document}